\documentclass[a4paper,UKenglish,cleveref, autoref, thm-restate]{lipics-v2021-ano}
\hideLIPIcs
\usepackage[disable,colorinlistoftodos,bordercolor=orange,backgroundcolor=orange!20,linecolor=orange,textsize=scriptsize]{todonotes}

\usepackage{amsmath}
\usepackage{mathtools}

\newtheorem{assumption}[theorem]{Assumption}

\title{Solving irreducible stochastic mean-payoff games and entropy games by relative Krasnoselskii-Mann iteration} %
\titlerunning{Relative Krasnoselskii-Mann iteration}%
\author{Marianne Akian}{INRIA and CMAP, École polytechnique, IP Paris, CNRS, France}{Marianne.Akian@inria.fr}{}{}

\author{St\'ephane Gaubert}{INRIA and CMAP, École polytechnique, IP Paris, CNRS, France}{Stephane.Gaubert@inria.fr}{}{}

\author{Ulysse Naepels}{École polytechnique, IP Paris, France}{Ulysse.Naepels@polytechnique.edu}{}{}
\author{Basile Terver}{École polytechnique, IP Paris, France}{Basile.Terver@polytechnique.edu}{}{}

\authorrunning{M. Akian, S. Gaubert, U. Naepels and B. Terver} %

\Copyright{Marianne Akian, St\'ephane Gaubert, Ulysse Naepels and Basile Terver} %

\ccsdesc[500]{Theory of computation~Algorithmic game theory} %

\keywords{Stochastic mean-payoff games, concurrent games, entropy games, relative value iteration, Krasnoselskii-Mann fixed point algorithm, Hilbert projective metric} %

\category{} %

\relatedversion{} %

\nolinenumbers %

\EventEditors{}
\EventDate{}
\EventLocation{}
\EventLogo{}
\SeriesVolume{}
\ArticleNo{}

\renewcommand{\geq}{\geqslant}
\renewcommand{\leq}{\leqslant}

\newcommand{\mbar}{\underline{m}}
\newcommand{\edges}{\mathscr{E}}
\newcommand{\vertexset}{\mathscr{V}}
\newcommand{\vertexsetD}{{\mathscr{V}_D}}
\newcommand{\vertexsetT}{{\mathscr{V}_T}}
\newcommand{\vertexsetP}{{\mathscr{V}_P}}

\newcommand{\dunion}{\uplus}
\newcommand{\eweight}{m}
\newcommand{\NN}{\mathbb{N}}

\newcommand{\RR}{\mathbb{R}}
\newcommand{\R}{\mathbb{R}}

\newcommand{\calm}{\mathcal{M}}
\newcommand{\norm}[1]{\|#1\|}
\newcommand{\normH}[1]{\|#1\|_{\text{\rm H}}}
\newcommand{\shapley}{T}
\newcommand{\shapleytheta}{T_{\theta}}
\newcommand{\shapleymKM}{T_{\textrm{m},\vartheta}}
\newcommand{\diag}{\operatorname{diag}}

\newcommand{\irr}{k_{\mathrm{irr}}}
\newcommand{\uni}{k_{\mathrm{uni}}}

\newcommand{\FmKM}{F_{\textrm{m},\vartheta}}
\newcommand{\FmKMtilde}{\tilde{F}_{\textrm{m},\vartheta}}
\newcommand{\randpolpair}{\Pi_{\mathrm{R}}}
\newcommand{\randpolMin}{\Pi_{\mathrm{R}}^{\mathrm{Min}}}
\newcommand{\randpolMax}{\Pi_{\mathrm{R}}^\mathrm{Max}}
\newcommand{\purepolpair}{\Pi_{\mathrm{P}}}
\newcommand{\purepolMin}{\Pi_{\mathrm{P}}^\mathrm{Min}}
\newcommand{\purepolMax}{\Pi_{\mathrm{P}}^\mathrm{Max}}
\newcommand{\E}{{\mathbb E}}

\newcommand{\pmin}{p_{\min}}
\newcommand{\xapprox}{\tilde{x}}
\newcommand{\shapleyapp}{\tilde{T}}
\newcommand{\mytop}{\mathbf{t}}
\newcommand{\mybot}{\mathbf{b}}
\newcommand{\mcst}{\bar{\mathcal{M}}}
\newcommand{\chiu}{\overline{\chi}}
\newcommand{\chil}{\underline{\chi}}
\newcommand{\hilbert}[1]{\|#1\|_{\mathrm{H}}}

\newcommand{\exle}[2]{{}^{#2}#1} 
\newcommand{\Tbeta}{\exle{T}{\beta}}
\newcommand{\Optmin}{\mathrm{Opt}_{\mathrm{Min}}}
\newcommand{\Optmax}{\mathrm{Opt}_{\mathrm{Max}}}

\usepackage{hyperref}
\usepackage{algorithm}
\usepackage{algpseudocode}
\floatname{algorithm}{Algorithm}
\newcommand{\Input}{\textbf{input}}
\newcommand{\new}[1]{{\em #1}}

\begin{document}

\maketitle
\makeatother
\def\@oddfoot{
     }
\makeatletter

\begin{abstract}
  We analyse an algorithm solving stochastic mean-payoff games,
  combining the ideas of relative value iteration and of
  Krasnoselskii-Mann damping.  We derive parameterized complexity
  bounds for several classes of games satisfying irreducibility
  conditions.  We show in particular that an $\epsilon$-approximation
  of the value of an irreducible concurrent stochastic game can be
  computed in a number of iterations in $O(|\log\epsilon|)$ where the
  constant in the $O(\cdot)$ is explicit, depending on the smallest
  non-zero transition probabilities. This should be compared with a
  bound in $O(|\epsilon|^{-1}|\log(\epsilon)|)$ obtained by Chatterjee
  and Ibsen-Jensen (ICALP 2014) for the same class of games, and to a
  $O(|\epsilon|^{-1})$ bound by Allamigeon, Gaubert, Katz and Skomra
  (ICALP 2022) for turn-based games. We also establish parameterized
  complexity bounds for entropy games, a class of matrix
  multiplication games introduced by Asarin, Cervelle, Degorre, Dima,
  Horn and Kozyakin.  We derive these results by methods of
  variational analysis, establishing contraction properties of the
  relative Krasnoselskii-Mann iteration with respect to Hilbert's
  semi-norm.
\end{abstract}

\section{Introduction}
\subsection{Motivation and context}
Stochastic mean-payoff games are a fundamental class of zero-sum games, appearing
in various guises.
In {\em turn-based} games, two players play sequentially, alternating moves,
or choices of an action,
being aware of the previous decision of the other player. Turn-based
games with mean-payoff and finite state and action spaces are among the unsettled problems in complexity theory:
they belong to the complexity class NP $\cap$ coNP~\cite{condon,zwick_paterson}
but are not known to be polynomial-time solvable. We refer the reader
to the survey~\cite{andersson_miltersen} for more information on the different
classes of turn-based games.
In contrast, in {\em concurrent games}, at each stage, the two players choose simultaneously one action, being unaware of the choice of the other player at the same stage. Turn-based games are equivalent to a subclass of concurrent games (in which in each state, one of the two
players is a dummy). The existence of the value for concurrent stochastic
mean-payoff games is a celebrated result of Mertens and Neyman~\cite{mertens_neyman}. This builds on earlier results by Bewley and Kohlberg, connecting mean-payoff concurrent games with discounted concurrent games, by making the discount factor tend to $1$, see~\cite{bewley_kohlberg}.
Concurrent games are hard to solve exactly: the value is an algebraic
number whose degree may be exponential in the number of states~\cite{tsigaridas}.\todo{SG: stated the stronger lower bound of \cite{tsigaridas}}
Moreover, concurrent reachability games are
square-root sum hard~\cite{Etessami2008}.

Another class consists of {\em  entropy games}, introduced by
  Asarin, Cervelle, Degorre, Dima, Horn and Kozyakin 
as an interesting category of ``matrix multiplication games''~\cite{asarin_entropy}.
Entropy games capture a variety of applications, arising
in risk sensitive control~\cite{Howard-Matheson,anantharam},
portfolio optimization~\cite{Akian2001}, growth maximization and population dynamics~\cite{Sladky1976,rothblumwhittle,rothblum,zijmjota}.
Asarin et al.\ showed that entropy games belong to the class NP $\cap$ coNP, showing
an analogy with turn based games.
In \cite{akian:hal-02143807},
Akian, Gaubert, Grand-Cl\'ement and Guillaud showed that entropy games are actually special cases of stochastic mean-payoff games, in which action spaces are infinite sets (simplices), and payments are given by Kullback-Leibler divergences.

A remarkable subclass of stochastic mean-payoff games arises when imposing
{\em ergodicity} or {\em irreducibility} conditions. Such conditions entail
that the value of the game is independent of the initial state.
The simplest condition of this type requires that every pair of policies (positional strategies) of the two players induces an irreducible Markov chain. Then, the solution
of the game reduces to solving
a nonlinear eigenproblem of the form $T(u)=\lambda e + u$, in which $u\in \R^n$
is a non-linear eigenvector, $\lambda$ is a non-linear eigenvalue, which provides the value of the mean-payoff game, $e$ is the unit vector of $\R^n$,
and $T$ is a self-map of $\R^n$, the 
dynamic programming operator of the game, which we shall refer to as the ``Shapley'' operator.
In fact, Shapley originally introduced
a variant of this operator, adapted to the discounted case~\cite{shapley_stochastic}.
The undiscounted mean-payoff case was subsequently considered by Gillette~\cite{gillette}.
We refer the reader
to~\cite{sorin_repeated_games,RS01} for background on Shapley operators and on the ``operator approach'' to games,
and to~\cite{ergodicity_conditions} for a discussion of the non-linear eigenproblem. 

In the one-player case, White \cite{WHITE1963373} introduced {\em relative value iteration},
which consist in fixed point iterations up to
additive constants $\lambda_k\in\R$, i.e.\ $x_{k+1}= T(x_k)-\lambda_k e$.
This solves the non-linear eigenproblem $T(u)= \lambda e +u$ under
a primitivity assumption.
However, this assumption appears to be too restrictive in the light of
the classical Krasnoselkii-Mann algorithm~\cite{krasno,mann}, which
allows one to find a fixed point
of a nonexpansive self-map $T$ of a finite dimensional normed space, by constructing the ``damped'' sequence
$x_{k+1}=(1-\theta) T(x_k)+\theta x_k$, where $0<\theta<1$.
Indeed, 
it was proposed in~\cite{stott2020} to apply this algorithm to 
the non-linear eigenproblem $T(u)= \lambda e +u$, thought of as a fixed point
problem in the quotient vector space $\R^n/\R e$. 
We will refer to this algorithm
as the {\em relative Krasnoselskii-Mann} value iteration.
An  error bound in $O(1/\sqrt{k})$ was derived in~\cite{stott2020} for this algorithm,
as a consequence of a general theorem of Baillon and Bruck~\cite{baillonbruck},
and the existence of an asymptotic geometric convergence rate
was established in a special case. This left open
the question of obtaining stronger iteration complexity bounds, in a ``white box model'',
for specific classes of stochastic mean-payoff games.

\subsection{Contribution}
We apply the relative Krasnoselskii-Mann
value iteration algorithm to deduce complexity bounds for several classes
of stochastic games.
We consider in particular {\em unichain}
concurrent stochastic mean-payoff games, in which every pair of positional
strategies of the two players induces a unichain transition matrix
(i.e., a stochastic matrix with a unique final class).
We define $\pmin$ to be the smallest non-zero off-diagonal transition probability in the model.
\Cref{cor-complexity-concur} shows that the relative Krasnoselskii-Mann iteration yields an
$\epsilon$-approximation of the value of the game, after
$C (|\log\epsilon|))$ iterations. The factor $C$ has an essential
term of the form $k \theta^{-k}$, in which $k\leq n$ is
a certain ``unichain index'', which is equal to $1$ if all the transition
probabilities are positive, $\theta=\pmin/(1+\pmin)$, and $n$ denotes the number of states.
Then, we consider the special case of unichain turn-based games,
with rational transition probabilities whose denominator divides $M$.
\Cref{th-uni-turn} shows that optimal policies can be obtained after a number
of iterations of order $M^k$.
The main tool is~\Cref{contraction_concu}, which shows that a suitable iterate of the Shapley
operator of a unichain concurrent game is a contraction in Hilbert's seminorm.
This theorem is proved using techniques of variational analysis,
in particular we use a classical result of Mills~\cite{mills}, characterizing
the directional derivative of the value of a matrix game, and properties
of nonsmooth semidifferentiable maps.

Finally, we introduce a variant of the relative Krasnoselkii-Mann algorithm,
adapted to entropy games. \Cref{th-irr-entropy} shows that an irreducible entropy game can be solved
exactly in a time of order $(1+\mathcal{A}/\mbar)^k$ where $k\leq n$ is a certain
``irreducibility index'', 
$\mbar\geq 1$ is the smallest multiplicity of an off-diagonal transition,
and $\mathcal{A}$ is a measure of the {\em ambiguity} of the game.
In particular, we have $W\leq \mathcal{A}\leq n^{1-1/n}W$ where $W$ is the maximal multiplicity of a transition.
The proof exploits the Birkhoff-Hopf theorem, which states that a positive
matrix is a contraction in Hilbert's projective metric.

\subsection{Related work}

The algorithmic approach of stochastic mean-payoff games games satisfying irreducibility conditions
goes back to the work of Hoffman and Karp~\cite{hoffman_karp}, applying
policy iteration to solve turn-based games.
Chatterjee and Ibsen-Jensen~\cite{DBLP:journals/corr/ChatterjeeI14} studied
more generally the {\em concurrent} stochastic mean-payoff games, under appropriate conditions of ergodicity.
They showed in particular
that the problem of approximation of the value is in FNP, and that this approximation problem, restricted to turn-based ergodic games, is at least as hard as the decision
problem for simple stochastic games. They also showed
that value iteration provides and $\epsilon$-approximation
of the value of a concurrent stochastic game statisfying an irreducibility condition
in $O(\tau \epsilon^{-1}|\log\epsilon|)$ iterations, where
$\tau$ denotes a bound of the passage time between any two
states under an arbitrary strategy, see Theorem~18, {\em ibid.} A recent ``universal bound''
on value iteration by Allamigeon, Gaubert, Katz and Skomra~\cite[Th.~13]{icalp2022} entails an
improvement of this bound to $O(\tau \epsilon^{-1})$. \Cref{cor-complexity-concur}
further improves this bound to get $C |\log\epsilon|$.
However, the later result requires an unichain assumption,
whereas the assumption of~\cite[Th.~13]{icalp2022} is milder.

The question of computing the value of a concurrent discounted stochastic game
has been studied by Hansen, Kouck\'y, Lauritzen, Miltersen and Tsigaridas in~\cite{tsigaridas}, who showed, using semi-algebraic geometry techniques,
that an $\epsilon$-approximation of the value of a general concurrent game can be obtained
in polynomial time if the number $n$ of states is fixed.
The exponent of the polynomial is of order $O(n)^{n^2}$ and
it was remarked in~\cite{tsigaridas} that ``getting a better dependence
on $n$ is a very interesting open problem''. 
Boros, Gurvich, Elbassioni and Makino considered the notion of $\epsilon$-ergodicity
of a concurrent mean-payoff game, requiring that the mean-payoff of two initial states
differ by at most $\epsilon$. They provided a potential-reduction
algorithm allowing one to decide $\epsilon$-ergodicity, and to get an $\epsilon$-approximation
of the value, with a dependence in $\epsilon$ of order $\epsilon^{-O(2^{2n} n\max(|A|,|B|))}$, see~\cite{Boros2016}.
Attia and Oliu-Barton developed in~\cite{miguel}
a bisection algorithm, with a complexity bound polynomial in $|\log \epsilon|$ and
in $|A|^n$ and $|B|^n$ where $A,B$ are the action spaces.
In contrast to these three works, our approach only applies to the subclass of {\em unichain} concurrent games,
but its complexity has a better dependence in the number of states; in particular, the exponents in our bound is at most
$n$,
and the execution time grows only polynomially with the numbers of actions $|A|$ and $|B|$.
Moreover, our approach applies more generally to infinite (compact) action spaces (we only need an oracle
evaluating the value of a possibly infinite matrix game up to a given accuracy).

The analysis of relative value iteration, using contraction techniques, goes back to
the work of Federguen, Schweitzer and Tijms~\cite{FEDERGRUEN1978711},
dealing with the one-player and finite action spaces case, under a
primitivity condition. The novelty here is the analysis of the concurrent two-player
case, as well as the analysis of the effect of the Krasnoselskii-Mann damping,
allowing one to replace earlier primitivity conditions by a milder unichain condition.
Moreover, even in the one-player
case, our formula for the contraction rate given in \Cref{contraction_concu} improves
the one of~\cite{FEDERGRUEN1978711} (see~\Cref{rk-explain} for a comparison).

Our results of \Cref{sec-entropy} dealing with entropy games are inspired by the series of works~\cite{asarin_entropy,akian:hal-02143807,icalp2022}.
The subclass of ``Despot-free'' entropy games can be solved in polynomial time~\cite{akian:hal-02143807},
and it is an open question whether general entropy games can be solved in polynomial time.
The approach of~\cite{icalp2022} entails that one can get an $\epsilon$-approximation
of the value of an entropy game in $O(\epsilon^{-1})$ iterations, where the factor in the $O(\cdot)$ is %
exponential in the parameters of the game. This bound is refined here
to $O(|\log\epsilon|)$, in which the factor in the $O(\cdot)$ depends
on a measure of ``ambiguity'' -- but our approach requires
an irreducibility assumption.

\section{Preliminary results on Shapley operators}

Let $n$ be an integer.
A map $T : \RR^n \rightarrow \RR^{n} $ is said to be \textit{order-preserving} when: $\forall x,y \in \RR^n, x \leq y \implies T(x) \leq T(y)$, where $\leq$ denotes the standard partial order of $\R^n$.
It is \textit{additively homogeneous} when: $\forall x \in \RR^n, \forall \lambda \in \RR, T(x+ \lambda e ) = T(x) + \lambda e$ where $e$ is the vector of $\RR^n$ having 1 in each coordinate.
\begin{definition}\label{def-shapley}
    A map $T : \RR^n \rightarrow \RR^n $ is an (abstract) {\em Shapley operator} if it is order-preserving and additively homogeneous. 
\end{definition}
We will justify the terminology ``Shapley operator'' in the next section,
where we give concrete examples, arising as dynamic programming operators
of different classes of zero-sum repeated games.
We set $[n]\coloneqq \{1,\ldots n\}$.
For any $x\in \R^n$, 
we denote $ \textbf{t}(x) \coloneqq \max_{i \in [n]} x_i$ and $ \textbf{b}(x) \coloneqq \min_{i \in [n]} x_i $
(read ``top'' and ``bottom''). We define the \new{Hilbert's seminorm}
of $x$ by: $\normH{x} = \textbf{t}(x) - \textbf{b}(x)$.
Since $\normH{x}=0$ iff $x\in \R e$, we get that
$\normH{\cdot}$ is actually a norm on the quotient vector
space $\R^n/\R e$. 
We also notice that: $\norm{x}_{\infty} = \inf\{\lambda \in \RR_+ : -\lambda e \leq x \leq \lambda e \}$
and: $\norm{x}_{H} = \inf\{\beta - \alpha \in \RR_+ : \alpha, \beta \in \RR, \alpha e \leq x \leq  \beta e\}$.
It is easy to show, thanks to these expressions, that a Shapley operator $T$ is non-expansive (i.e., $1$-Lipschitz) for $\norm{\cdot}_{\mathrm{H}}$ and for $\norm{\cdot}_{\infty}$.
Then, %
 it induces a self-map $\overline{T}$
on the quotient vector space $\R^n/\R e$, sending
the equivalence class $x+\R e$ to $T(x)+\R e$,
and which is non-expansive.
\begin{definition}
  We define the \emph{escape rate} $\chi(T)$ of a Shapley operator $T$ as
  $\lim_{k\to\infty} k^{-1}T^k(v)\in \R^n$, where $v$ in $\R^n$.  The lower and upper escape rates are defined respectively by
  \(
  \chil(T)=\lim_k k^{-1} \mybot (T^k(v))\) and
  \(  \chiu(T)=\lim_k k^{-1} \mytop (T^k(v))\).
\end{definition}
Since $T$ is nonexpansive in the sup-norm, the existence and the values
of these limits are independent of the choice of $v\in \R^n$.
In general, the escape rate $\chi(T)=\lim_{k\to\infty} k^{-1}T^k(v)$ may not exist,
but a subadditive argument shows that the lower and upper escape rates always exist, see e.g.~\cite{arxiv1}.
A fundamental tool to establish the existence of the escape rate is to consider
the following ergodic equation.
\begin{definition}\label{ergodic}
    We say that the {\em ergodic equation} has a solution when there exists $\lambda \in \RR$ and $u \in \RR^n$ such that : $T(u) = \lambda e + u$.  
\end{definition}
\begin{observation}\label{existence_escape_rate}
\label{lemme_encadrement}
  If the above ergodic equation is solvable, then $\chi(T)=\lambda e$.
  More generally, if $\alpha e +v  \leq T(v) \leq \beta e+ v $ for
  some $v\in \R^n$ and $\alpha,\beta\in\R$, then
  $\alpha \leq \chil(T)\leq \chiu(T) \leq \beta$.
\end{observation}
\begin{proof}  
 By an immediate induction, and as $T$ is order-preserving and additively homogeneous we have : $ k\alpha e +v \leq T^k(v) \leq k \beta e + v $.
 Then, $ k\alpha +\mybot(v)  \leq \mybot(T^k(v))\leq \mytop(T^k(v))\leq k\beta  + \mytop(v)$. Dividing by $k$ and letting $k$ tend to infinity, we obtain the second statement.
\end{proof}
We are inspired by the following observation from fixed point theory,
proved in~\Cref{sec-proof-obs}.
\begin{observation}\label{obs-contract}
  Suppose now that $T^q$ is $\gamma$-contraction in Hilbert's seminorm $\|\cdot\|_H$,
  for some $q\geq 1$ and $0<\gamma<1$. Then, the ergodic equation
  is solvable.
\end{observation}
Shapley operators include (finite dimensional) Markov operators, which
are of the form $T(x)=Mx$, where $M$ is a $n\times n$ stochastic matrix
(meaning that $M$ has nonnegative entries and row sums one). In this
case, an exact formula is known for the contraction rate.
In fact, one can consider the operator norm of $M$, thought
of as a linear map acting on the quotient vector space $\R^n/\R e$,
\(
\norm{M}_{\mathrm{H}} = \underset{u \notin \RR e }{\sup}\frac{\norm{Mu}_{\mathrm{H}}}{\norm{u}_{\mathrm{H}}}
\).
\begin{theorem}[Corollary of~\cite{Dobrushin1956I}]\label{th-dob}
\(
  \hilbert{M} =  \delta(M) \coloneqq
  1 - \min_{1\leq i<j\leq n} \big\{\sum_{k \in [n]} \min(M_{ik},M_{jk})\big\}
  \).
\end{theorem}
The derivation from~\cite{Dobrushin1956I} is recalled in~\Cref{sec-explain-dobrushin}.
The term $\delta(M)$ is known as {\em Dobrushin ergodicity coefficient}.

\section{Two classes of zero-sum two-player repeated games}
We next recall the defintition and basic properties of two classes of zero-sum two-player games with finite state
spaces.
More details can be found in \cite{sorin_repeated_games} for stochastic games and in \cite{asarin_entropy,akian:hal-02143807} for entropy games.
\subsection{Concurrent repeated zero-sum stochastic two-player games}
\label{repeated_games}
We assume that the state space is equal 
to $[n] = \{1,\dots,n\}$. We call the two players
``Min'' and ``Max''.
The game is specified
by the following data. For every state $i\in [n]$, 
we are given two non-empty compact sets $A(i)$ and $B(i)$,
representing the admissible actions of players Min and Max,
respectively.
For every $i\in [n]$ and every choice
of actions $(a,b)\in A(i)\times B(i)$, we are given
a real number $r_i^{ab}$, representing
an instantaneous payment, and a stochastic vector
$P_i^{a,b}=(P_{ij}^{ab})_{j\in[n]}$,
meaning that
$P_{ij}^{ab}\geq 0$ and that $\sum_{j\in [n]}P^{ab}_{ij}=1$.
We assume that the functions $(a,b)\mapsto r_i^{ab}$
and $(a,b) \mapsto P_i^{a,b}$ are continuous.

The concurrent game is played in successive stages, starting
from a known initial state $i_0$ at stage $0$.
We denote by $a_k$ and $b_k$ the actions selected by
Players Min and Max at stage $k$, respectively,
and by $i_k$ the state at this stage.
The history until stage $k$ consists of the sequence
$H_k=((i_\ell,a_\ell,b_\ell)_{0\leq \ell < k}, i_k)$.
A randomized strategy of Player Min (resp.\ Max)
is a collection of measurable functions assigning to
every history $H_k$ a probability measure $\alpha_k$
(resp.\ $\beta_k$)
on the compact set $A(i_k)$ (resp.\ $B(i_k)$). At stage $k$,
being informed of the history $H_k$ up to this stage,
Player Min draws
a random action $a_k$ according to the probability measure $\alpha_k$,
and similarly, Player Max draws a random action $b_k$ according
to the probability measure $\beta_k$. Then, Player Min
makes to Player Max an instantaneous payment of $r_{i_k}^{a_k,b_k}$,
and the next state $i_{k+1}$ is drawn randomly according
to the probability measure $(P_{i_k,j}^{a_kb_k})_{j\in [n]}$ on the state space $[n]$,
i.e., the conditional probability that $i_{k+1}=j$,
given the history $H_k$ and actions $a_k,b_k$,
is given by $P_{i_k,j}^{a_kb_k}$. We shall say
that a strategy is {\em pure} or {\em deterministic} if the action
of the player is chosen as a deterministic
function of the history.
We denote by $\sigma_k$ (resp.\ $\tau_k$) the strategy of Player Min (resp.\ Max) at stage $k$, %
and denote by $\sigma$ and $\tau$ the sequences $(\sigma_k)_{k\geq 0}$ and
$(\tau_k)_{k\geq 0}$.
In this way, to any initial state $i_0\in [n]$ and any pair of strategies $(\sigma,\tau)$ 
of the two players %
is associated the infinite random sequence $(i_k,a_k,b_k)_{k\geq 0}$.
We denote by $\E^{\sigma,\tau}_{i_0}$ the expectation with respect to this process.

We shall need to consider special classes
of strategies.
A {\em Markovian} or {\em positional} strategy is a strategy such that $\sigma_k$ (resp.\ $\tau_k$) depends only on the current state $i_k$. %
Such a positional strategy $\sigma_k$ or $\tau_k$ is also called a \new{policy}.
It is said to be stationary if this policy is independent of $k$.
We shall denote by $\Delta_{A(i)}$ (resp.\ $\Delta_{B(i)}$) the set
of probability measures on $A(i)$ (resp.\ $B(i)$).
We denote by $\purepolMin$ (resp. $\randpolMin$), the set of pure 
(resp.\ randomized) policies of Min. It is in bijection (and will be identified)
with  $\prod_{i\in [n]} A(i)$ (resp.\  $\prod_{i\in [n]} \Delta_{A(i)}$).
Indeed, if $\sigma\in\purepolMin$, then $\sigma$ is identified
with $(\sigma(1),\ldots, \sigma(n))$ where 
$\sigma(i)\in A(i)$ is the action Min is choosing in state $i$
according to the policy $\sigma$.
Similarly, if $\alpha\in\randpolMin$, 
$\alpha$ is identified with $(\alpha_1,\ldots, \alpha_n)$,
where for $i\in [n]$, and $a\in A(i)$,
$d\alpha_i(a)$ is the probability that Min 
chooses the action $a$ in state $i$, according to the policy $\alpha$.
Finally, a pure policy $\sigma$ of Min is identified to
a randomized policy consisting of Dirac measures,
so that $\purepolMin\subset \randpolMin$.
We use the same notations and identifications for Max.
We shall also denote by  $\purepolpair =\purepolMin \times \purepolMax$ and $\randpolpair =\randpolMin \times \randpolMax$
the spaces of pairs of policies.

Given
an initial state $i_0$ 
and a pair of strategies $(\sigma,\tau)$ of the two players,
the expected payment received by Player Max in horizon $N$ is defined
by 
\[ J^N_{i_0}(\sigma, \tau)
\coloneqq \E^{\sigma,\tau}_{i_0}\left[ \sum_{k=0}^{N-1} r_{i_k}^{a_k,b_k}\right]
\enspace .\]
We shall denote by $J^N(\sigma, \tau)$ the vector of $\R^n$ 
with the above $i_0$ entry, for each $i_0\in [n]$.
The finite horizon game has a value $v^N\in\R^n$ and has a pair of
optimal (randomized) strategies $(\sigma^*, \tau^*)$,  meaning that
\begin{equation}\label{saddle}
    J^N(\sigma^*, \tau) \leq v^N=J^N(\sigma^*, \tau^*) \leq J^N(\sigma, \tau^*)
\enspace ,
\end{equation}
for all pairs $(\sigma,\tau)$ of strategies,
see~\cite{sorin_repeated_games}.
Moreover, one can choose the pair of optimal strategies $(\sigma^*, \tau^*)$ to be positional,
that is $(\sigma^*_k,\tau^*_k)\in \randpolpair$ for all $k\leq N$ (but it
generally depends on $k$ and $N$).
These optimal strategies can be obtained by using the dynamic programming equation of the 
game, as follows.

For any $i,j\in [n]$, $\alpha_i\in \Delta_{A(i)}$ and $\beta_i\in \Delta_{B(i)}$,
let us denote 
\begin{align} r^{\alpha_i,\beta_i}_i = \int_{A(i)\times B(i)}  r_i^{a,b} d\alpha_i(a)d\beta_i(b)\quad \text{and} \quad 
  P_{i,j}^{\alpha_i,\beta_i} = \int_{A(i)\times B(i)}  P_{i,j}^{a,b}d\alpha_i(a)d\beta_i(b)\enspace .\label{e-int}
\end{align}
This extends the functions $(a,b)\mapsto r_i^{a,b}$ and
$(a,b)\mapsto P_{i,j}^{a,b}$ from $A(i)\times B(i)$ to 
$\Delta_{A(i)}\times \Delta_{B(i)}$.
We then define the Shapley operator $T$ of the concurrent game as
the map $T: \R^n\to \R^n$ such that
\begin{align}
    T_i(v) =  \min_{\alpha_i \in \Delta_{A(i)} } \max_{\beta_i \in \Delta_{B(i)} }
    \Big( r_i^{\alpha_i,\beta_i} + \sum_{j\in [n]} P_{ij}^{\alpha_i,\beta_i} v_j\Big),
    \;\text{ for } i\in [n], \; v\in \R^n \enspace .\label{e-def-shapley-concurrent}
\end{align}
Note that in the above expression the infimum and supremum commute, owing to the compactness
of action spaces, and continuity assumptions on the functions $(a,b)\mapsto r_i^{a,b}$
and $(a,b)\mapsto P_{ij}^{ab}$ (this follows from Sion's minimax theorem).
\todo{BT : plutôt théorème de point selle de Sion ? SG. von neumann suffit pour des matrices finies. dans le cas compact on invoquera sion. SG. merci!}
Moreover, the operator $T$ satisfies the properties of \Cref{def-shapley}.
Then, the value of the concurrent game in finite horizon is obtained from the 
recurrence equations:
\( v^0 = 0, \quad v^N = T(v^{N-1})\).
Moreover, optimal strategies of the game when the remaining time is $k< N$
(or at stage $N-k$) are obtained by choosing 
optimal policies $\alpha$ and $\beta$ 
with respect to the vectors $v^{k}$, that is such that
$\alpha_i$ and $\beta_i$ are optimal in the expression of
$T_i(v^k)$ in \eqref{e-def-shapley-concurrent}.

We now describe the {\em mean-payoff game}, which is obtained by considering
the Cesaro limit of the payoff as the horizon $N$ tends
to infinity. More precisely, we set:
\[ \chi^+_{i_0}(\sigma, \tau)\coloneqq \limsup_{N\to\infty} N^{-1}{J^N_{i_0}(\sigma, \tau)}\quad \chi^-_{i_0}(\sigma, \tau)\coloneqq \liminf_{N\to\infty} N^{-1}{J^N_{i_0}(\sigma, \tau)}\enspace.\]
We shall say that the game with mean-payoff has a value $\chi^*\in\R^n$
if for all $\varepsilon>0$, there exists
strategies $\sigma^\varepsilon,\tau^\varepsilon$  of the two players
which are $\varepsilon$-optimal,
meaning that for every strategies $\sigma$ and $\tau$,
\(   -\varepsilon e+ \chi^+ (\sigma^\varepsilon, \tau) \leq \chi^*
\leq  \chi^- (\sigma, \tau^\varepsilon)+\varepsilon e\).
Mertens and Neyman~\cite{mertens_neyman}, building
on a result of Bewley and Kohlberg~\cite{bewley_kohlberg}, showed
that when the action spaces $A(i)$ and $B(i)$ are finite,
the mean-payoff game has a value (actually, in a stronger
{\em uniform} sense). Moreover, the value coincides
with the escape rate of the Shapley operator,
i.e., $\chi^*=\lim_k T^k(0)/k$. A counter-example
of Vigeral shows that these properties do not carry
over to the case of general compact action spaces~\cite{Vigeral2013}.

One particular case that will interest us 
is when the ergodic equation is solvable, that is when there exists $\lambda\in \R$ and $v\in \R^n$ such that $T(v)=\lambda e +v$.
In that case, $\chi^*=\lambda e$ and there exists optimal randomized strategies
for the two players
which are both positional and stationary. Such a pair of strategies
is obtained by choosing a pair $(\alpha,\beta)$ 
of optimal policies with respect to $v$,
meaning optimal in the expression of $T(v)$ in \eqref{e-def-shapley-concurrent}.
We shall see that the ergodic equation is always solvable
under a unichain condition, even in the case of compact action spaces
(\Cref{th-exists-unichain}).%

A remarkable subclass of concurrent games consists
of {\em turn-based} games. Then, the actions spaces
$A(i)$ and $B(i)$ are required to be {\em finite}, and
for every state $i\in [n]$,
we assume that either $A(i)$ or $B(i)$ is a singleton.
In other words, there is a bipartition $[n]=I_{\text{Min}}\dunion I_{\text{Max}}$
of the set of states, so that in every state $i\in I_{\text{Min}}$ (resp.\ $I_{\text{Max}}$), Min
(resp.\ Max) is the only player who has to take a decision.
Then, the Shapley operator of the game reduces to
\(%
T_i(v) = \min_{a \in A(i)} \max_{b \in B(i)}\big(r_i^{a,b} + \sum_j P_{i,j}^{a,b} v_j\big)
\),
for $i\in [n]$,
where again the min and max commute, because in every $i\in [n]$,
either the min or the max is taken over a set reduced
to a singleton.
Hence, the set of optimal policies with respect to a vector $v$ contains
pure policies and is obtained by taking $\sigma(i)=a$ 
and $\tau(i)=b$ optimal in the previous expression of $T_i(v)$.
This yields pure optimal policies for finite horizon turn based 
stochastic games, and for mean-payoff turn based stochastic games 
for which the ergodic equation is solvable. 
The existence of pure optimal policies for turn-based mean-payoff stochastic games
was shown by Liggett and Lippman \cite{LL69}.
 An illustrative example is given in~\Cref{sec-example}.
\subsection{Entropy games}\label{sec-entropy-games}
\todo{SG: I changed the notation of entropy games, following~\cite{icalp2022}}
Entropy games were introduced in \cite{asarin_entropy}. 
We use here the slightly more general model of~\cite{akian:hal-02143807,icalp2022}, to which we
refer for background. 
An {\em entropy game} is a turn-based game played on a (finite) digraph $(\vertexset,\edges)$, with two players,
called ``Despot'' and ``Tribune'', and an additional non-deterministic player, called ``People''.
We assume the set of vertices $\vertexset$ has a non-trivial partition: $\vertexset=\vertexsetD \dunion \vertexsetT \dunion \vertexsetP$. Players Despot, Tribune, and People control the states in $\vertexsetD$, $\vertexsetT$ and $\vertexsetP$ respectively, and they alternate their moves, i.e., $\edges\subset
(\vertexsetD \times \vertexsetT )\cup (\vertexsetT \times \vertexsetP) \cup (\vertexsetP \times \vertexsetD)$. 
We suppose that every edge $(p,d)\in \edges$ with $p\in \vertexsetP$ and $d\in \vertexsetD$
is equipped with a {\em multiplicity} $\eweight_{pd}$ which is a (positive)
natural number. For simplicity of exposition, we shall define here the value of an entropy game using
only pure policies (stationary positional strategies). More precisely, a policy $\sigma$ of Despot
is a map which assigns to every node $d\in \vertexsetD$ a node $t$ such that $(d,t)\in \edges$.
Similarly, a policy $\tau$ of Tribune is a map which assigns to every node $t\in \vertexsetT$ a
node $p\in \vertexsetP$. We denote by $n$ the cardinality of $\vertexsetD$.
Such a pair of policies determine a $n\times n$ matrix $M^{\sigma,\tau}$, such that
\(
M^{\sigma,\tau}_{d,d'} = m_{\tau(\sigma(d)),d'}
  \).
  Given an initial state $\bar{d}\in \vertexsetD$, we measure the ``freedom'' of Player People by
  the limit
  $R(\sigma,\tau)\coloneqq \lim_{k\to\infty} ((M^{\sigma,\tau})^ke)_{\bar{d}}^{1/k}$. 
A pair of policies determine a subgraph $\mathcal{G}^{\sigma,\tau}$,
  obtained by keeping only the successor prescribed by $\sigma$ for every node of $\vertexsetD$,
  and similarly for $\tau$ and $\vertexsetT$. 
 Then, the ``freedom'' of people is precisely the geometric growth rate of the
  number of paths of length $3k$ starting from node $\bar{d}$, counted 
with multiplicities,
as $k\to \infty$. In general, the graph $\mathcal{G}^{\sigma,\tau}$
may have several
strongly connected components,
  and it is observed in~\cite{icalp2022} that $R(\sigma,\tau)$ coincides
  with the maximal spectral radii of the diagonal blocks of the matrix
  $M^{\sigma,\tau}$ corresponding to the strongly connected components to which the initial state $\bar{d}$ has access
  in $\mathcal{G}^{\sigma,\tau}$.
  In an entropy game, Despot wishes to minimize the freedom of People, whereas
  Tribune (a reference to the magistrate of Roman republic) wishes to maximize it.
  It is shown in~\cite{akian:hal-02143807} that the entropy game
  has a value in the space of positional strategies, meaning that there exists policies $\sigma^*,\tau^*$,
  such that $R(\sigma^*,\tau)\leq R(\sigma^*,\tau^*)\leq R(\sigma,\tau^*)$ for all policies
  $\sigma,\tau$.

The dynamic programming operator of an entropy game is the self-map
$F$ of $\mathbb{R}_{>0}^n$ given by
\(%
F_d(x) = \min_{t\in \vertexsetT, (d,t)\in\edges}\; \max_{p\in \vertexsetP, (t,p)\in\edges}  \sum_{d'\in \vertexsetD, (p,d')\in\edges} m_{p,d'} x_{d'}\),
for $d\in\vertexsetD$.
Then, the operator $T  \coloneqq \log \circ F \circ \exp$ is a Shapley operator.
It is shown in~\cite{akian:hal-02143807} that the value of the entropy game with initial state $\bar{d}$ is given
by the limit $\lim_{k\to \infty} [F^k(e)]_{\bar{d}}^{1/k}$. 

\section{The unichain property}
Recall that to every $n\times n$ nonnegative matrix $M$ is associated a digraph
with set of nodes $[n]$, such that there is an arc from $i$ to $j$ if $M_{ij}>0$.
The matrix is {\em irreducible} if this digraph is strongly connected.
It is {\em unichain} if this digraph has a unique final strongly connected
component (a strongly components is {\em final} if any path starting from
this component stays in this component). The property of unichainedness
is sometimes referred to as {\em ergodicity} since a stochastic matrix
is unichain iff it has only one invariant measure, or equivalently,
if the only {\em harmonic vectors} (i.e. the solutions
$v$ of $Mv=v$) are the constant vectors, see the discussion
in Theorem 1.1 of \cite{ergodicity_conditions}, and the references
therein.

Given a pair $(\sigma,\tau)\in \purepolpair$ of pure policies,
we define the stochastic matrix:
\( %
P^{\sigma,\tau} = (P_{i,j}^{\sigma(i),\tau(i)})_{i,j\in [n]} \).

\begin{definition}
  We say that a game is {\em unichain} (resp.\ {\em irreducible})
  if for all pairs of pure policies  $\sigma,\tau$,
the matrix  $P^{\sigma, \tau}$ is unichain (resp.\ irreducible).
\end{definition}

  \begin{definition}\label{def-invariant}
    We say that a subset $S$ of the states is {\em closed} under the action of a matrix $P^{\sigma, \tau}$ if, starting from a state $s\in S$ and playing according
    to the policies $\sigma$ and $\tau$, the next state is still in $S$.
\end{definition}
\begin{remark}
    If $S$ is a set closed under the action of an unichain matrix $P^{\sigma, \tau}$, then $S$ contains the final class of this matrix. 
\end{remark}
\begin{remark}
The final class does not have to be the same for all pairs  $\sigma, \tau$ of policies in our definition of unichain games.
  \end{remark}
The following theorem, proved in~\Cref{sec-proof-th-exists-unichain}, addresses the issue of the existence of a solution to the ergodic equation in the case of a unichain game.
\begin{theorem}
  \label{th-exists-unichain}
  Let $T$ be the Shapley operator of a unichain concurrent stochastic game. Then, there exists a vector $v\in \R^n$ and $\lambda\in \R$ such that $T(v)=\lambda e +v$. Moreover, there exists a pair of optimal (randomized) positional strategies, obtained by selecting actions that achieve the minimum and maximum in the expression
  of $[T(v)]_i$, for each state $i\in[n]$.\todo{SG: proof of the last statement to be added}
\end{theorem}
\section{Relative value iteration}
Relative value iteration was
introduced in~\cite{WHITE1963373} to solve one player
stochastic mean-payoff games (i.e., average
cost Markov decision processes).
The ``vanilla'' value iteration
algorithm consists in computing the sequence $x_{k+1}=T(x_k)$,
starting from $x_0=0$.  Then, $x_k$ yields the value vector
of the game in horizon $k$, an so, we expect $x_k$
to go to infinity as $k \to \infty$. The idea of {\em relative}
value is to renormalize the sequence by additive constants.
We state in \Cref{RVI} a general version of relative value
iteration, allowing for {\em approximate}
dynamic programming oracles. This will allow us
to obtain complexity results in the Turing model of computation,
by computing a rational approximation of the value of the Shapley operator $T(x)$ at a given rational
vector $x$ up to a given accuracy.
\begin{algorithm}[hbpt]
\caption{Relative value iteration in approximate arithmetics}\label{RVI}
\begin{algorithmic}[1]
  \State \Input: A final requested numerical precision $\epsilon>0$ and a parameter $0< \eta\leq \epsilon /3$.
  An oracle $\shapleyapp$ which provides an $\eta$-approximation in the sup-norm of a Shapley operator $T$.
\State $ x \coloneqq  0 \in \RR^n$

\Repeat \State $x \coloneqq  \shapleyapp(x) - \mytop(\shapleyapp(x))e $
\Until { $ \norm{x-\shapleyapp(x)}_{\mathrm{H}} \leq \epsilon/3 $}
\State $\alpha \coloneqq  \mybot(\shapleyapp(x)-x); \beta \coloneqq  \mytop(\shapleyapp(x)-x)$\\

\Return $x,\alpha,\beta$
\Comment{The lower and upper escape rates of $T$ are included in the interval $[\alpha-\epsilon/3, \beta+\epsilon/3]$, which is of width at most $\epsilon$}
\end{algorithmic}
\end{algorithm}
\begin{theorem}\label{th-terminates}
  Suppose that $T$ is a Shapley operator. Then,
  \begin{enumerate}
    \item When it terminates, \Cref{RVI} 
      returns a valid interval of width at most $\epsilon$ containing
      the lower and upper escape rates of $T$.
\item If there is an integer $q$ and a scalar $0<\gamma<1$ such that
  $T^q$ is a $\gamma$-contraction in Hilbert's seminorm,
  and if $\eta$ is chosen small enough, in such a way that
  $\eta (12+24q/(1-\gamma)) \leq  \epsilon$, then 
  \Cref{RVI} terminates in at most
  \(
     q (\log\|T(0)\|_H+ \log 6 +|\log\epsilon| )/
  |\log \gamma|
  \)
  iterations.
  \end{enumerate}
\end{theorem}
The proof is given in \S\ref{sec-proof-RVI},
it exploits the nonexpansiveness
of the operator $T$ in Hilbert's seminorm.

\section{Krasnoselskii-Mann damping}

We shall see that for turn-based or concurrent games, it is useful to replace
the original Shapley operator by a Krasnoselskii-Mann damped version of this operator.
This will allow the relative-value iteration algorithm to converge under milder conditions.
\begin{definition}
    If $T$ is a Shapley operator, we define $\shapleytheta = \theta I + (1-\theta)T$ where $I$ is the identity operator.
\end{definition}
We will call {\em Krasnoselskii-Mann operator} the $\shapleytheta$ operator. It is easy to show that it is also a Shapley operator.
The following observation relates the ergodic constant of a damped Shapley operator
with the ergodic constant of the original Shapley operator.
\begin{lemma}
  Let $T$ be a Shapley operator, $u\in \R^n$ and  $\lambda\in \R$. Then,
  $T(u) = \lambda e + u$ if and only if $\shapleytheta (u) = (1-\theta)\lambda e + u$.
  In particular, $\chi(T) = (1-\theta)^{-1} \chi(\shapleytheta)$ holds as soon as the ergodic
  equation $T(u) = \lambda e +u$ is solvable.
\end{lemma}
\begin{proof}
  The equivalence is straightforward, and $\chi(T) = (1-\theta)^{-1} \chi(\shapleytheta)$
  follows from Obs.~\ref{existence_escape_rate}.
\end{proof}
We consider the iteration $x_{k+1} = T_\theta(x_k) - \mathsf{t}(T_\theta(x_k)) e$, obtained by applying
relative value iteration (as in~\Cref{RVI}) to the Krasnoselskii-Mann operator $T_\theta$, with
an arbitrary initial condition $x_0\in \R^n$.
Ishikawa showed that the ordinary Krasnoselskii-Mann iteration applied to a nonexpansive self-map of a finite dimensional normed space does converge, as soon as a fixed point exists~\cite{ishikawa}.
This entails the following result. 
\begin{theorem}[Compare with \cite{stott2020}]\label{th-stott}
  Let  $T: \R^n\to \R^n$ be a Shapley operator, and $0<\theta<1$. Then, the sequence $x_k$
  obtained by applying relative value iteration to the Krasnoselskii-Mann operator $T_\theta$
  converges 
if and only if 
$\exists u\in \R^n, \lambda\in \R$, $T(u) = \lambda e + u$. 
\end{theorem}
A multiplicative variant
of this result was proved in Theorem~11 of~\cite{stott2020}. We provide
the proof in \S\ref{appendix-proofofcvg} for completeness.

\section{Contraction properties of unchain games under pure policies}
\label{sec-unichain}
  We define the following parameter, representing the minimal value
  of a non-zero off-diagonal transition probability,
  \( \pmin \!= \!\underset{i,j \in [n], i\neq j, \, (a,b)\in A\times B}{\min} \{P_{i,j}^{a,b} : P_{i,j}^{a,b} >0 \} \),
and set
\(
\theta \coloneqq \pmin/(1+\pmin)\).
For every pair of policies $\sigma,\tau$ of the two players, we set
\(
Q^{\sigma,\tau}\coloneqq \theta I+(1-\theta)P^{\sigma,\tau} \).
For any sequence of pairs of pure policies $\sigma_1, \tau_1, \ldots , \sigma_k, \tau_k$,
we define, for all $i\in [n]$, $S_{i}(\sigma_1,\tau_1,\dots,\sigma_k,\tau_k)\coloneqq  \{j\mid  [Q^{ \sigma_1, \tau_1} \dots Q^{ \sigma_k, \tau_k}]_{i j}>0\}$.
\begin{lemma}\label{lemma-uni}
  Suppose a concurrent game is unichain. Then, there is an integer $k\leq n$ such that for
  all $i_1,i_2\in [n]$, and for all sequences of pairs of pure policies $\sigma_1, \tau_1, \ldots , \sigma_k, \tau_k$,
  $S_{i_1}(\sigma_1,\tau_1,\dots,\sigma_k,\tau_k)\cap S_{i_2}(\sigma_1,\tau_1,\dots,\sigma_k,\tau_k)\neq\varnothing$.
\end{lemma}
We call the {\em unichain index} of the game, and denote by $\uni$ the smallest integer
$k$ satisfying the property of~\Cref{lemma-uni}.
  Similarly, we call {\em irreducibility index} of an irreducible game, and denote by $\irr$, the smallest integer $k$
  such that for every sequence of pure policies $\sigma_1, \tau_1, \cdots,  \sigma_k, \tau_k$,
  the matrix $Q^{ \sigma_1, \tau_1} \dots Q^{ \sigma_k, \tau_k}$ is positive. We have
  $1\leq \uni \leq \irr$.

  The following result will allow us to obtain a geometric contraction rate.
  The proofs of this theorem and of the next proposition shows in particular
  that $\uni\leq n$ if the game is unichain
  and $\irr \leq n$ if the game is irreducible.

\begin{theorem}\label{thm_unichain}
  Let us suppose that a concurrent game with $n$ states is unichain, with unichain index $k=\uni$.
  Then, for all sequences $\sigma_1, \tau_1, \cdots,  \sigma_k, \tau_k$
  of pairs of pure policies of the two players,
\(%
\norm{Q^{ \sigma_1, \tau_1} \dots Q^{ \sigma_k, \tau_k}}_{\mathrm{H}}  \leq 1- \theta^k \).
\end{theorem}
The following proposition improves the bound on the contraction rate provided by~\Cref{thm_unichain}, in the special
case of irreducible games.
\begin{proposition}\label{prop-irr}
    Let us suppose that a concurrent game with $n$ states is irreducible, and let $k=\irr$ be the irreducibility index of the game.  Then, for all sequences $\sigma_1, \tau_1, \cdots,  \sigma_k, \tau_k$
    of pairs of pure policies of the two players,
\(%
\norm{Q^{ \sigma_1, \tau_1} \dots Q^{ \sigma_k, \tau_k}}_{\mathrm{H}}  \leq 1- n \theta^k\).
\end{proposition}
The proofs of \Cref{thm_unichain} and~\Cref{prop-irr} are provided in~\Cref{sec-proof-unichain}.

\section{Solving concurrent and turn-based games by relative Krasnoselskii-Mann iteration}
We first establish a general bound for concurrent unichain games. Recall that $\theta=\pmin/(1+\pmin)$.
\begin{theorem}\label{contraction_concu}
  Let $\shapleytheta$ be the Krasnoselskii-Mann operator of a concurrent and unichain game, and
  Then, $\shapleytheta^{k}$ is a contraction in Hilbert's seminorm, with rate bounded by $1 - \theta^{\uni}$.
  Moreover, if the game is irreducible, the same rate is bounded by $1-n\theta^{\irr}$.\todo{SG: added the irreducible case}
\end{theorem}

The theorem is proved in \Cref{proof-km-concur}.
Combining this result with~\Cref{th-terminates},
we obtain the following result, in which 
we denote by $\|r\|_\infty\coloneqq \max_{i,a,b} |r_{i}^{ab}|$ the sup-norm of the payment function.
\begin{corollary}\label{cor-complexity-concur}
  Let $T$ be the Shapley operator of a concurrent unichain game,
  and $\epsilon \in (0,1)$. \Cref{RVI}, applied to the Krasnoselskii-Mann operator $\shapleytheta$, with the precision
        $\eta$ prescribed in~\Cref{th-terminates},
  provides a $\epsilon$-approximation of the value of the game in at most
  \(
  (|\log(\epsilon)| + \log(1-\theta) + \log 12 + \log \|r\|_\infty )\uni \theta^{-\uni} 
  \)
  iterations.
\end{corollary}

We now consider the special case of turn-based games. Then, the value is a rational number,
and there are optimal pure policies. We now apply our approach
to compute exactly the value and to find optimal pure policies.

\begin{assumption}\label{a-integer}
We now assume that the probabilities $P_{i,j}^{a,b} $ are rational numbers with a common denominator
denoted by $M$. We also assume that the payments $r_i^{a,b}$ are {\em integers}.
\end{assumption}

\begin{lemma}[{Coro.~of~\cite{skomra_bounds}}]\label{mateusz}
Let $P$ be a
 $n\times n$ unichain matrix whose entries
are rational numbers with a common denominator $M$.
Then, the entries of the unique invariant measure of $P$ are rational numbers of denominator
at most $nM^{n-1}$.
  \end{lemma}

When \Cref{RVI} halts, returning a vector $x\in \R^n$, we select two pure policies $\sigma^*$ and $\tau^*$ that are optimal with respect to $x$, 
meaning that, for $i\in [n]$, we have:
\begin{align}\label{e-selrule}
  T_i(x) = \max_{b \in B(i)}
\Big( r_i^{\sigma^*(i),b} + \sum_{j\in [n]} P_{ij}^{\sigma^*(i),b} x_j\Big)
=\min_{a \in A(i) } \Big( r_i^{a,\tau^*(i)} + \sum_{j\in [n]} P_{ij}^{a,\tau^*(i)}
x_j\Big) \enspace.
\end{align}

\begin{theorem}\label{th-uni-turn}
  Consider a unichain turn-based stochastic game satisfying~\Cref{a-integer}.
  Let us choose $\epsilon= (1-\theta)(n^2 M^{2(n-1)})^{-1}$,
  so that Algorithm~\ref{RVI} applied to $\shapleytheta$ runs in at most\todo{SG: I fixed/optimized the choice of $\epsilon$}
  \begin{align}
  (\log(1-\theta) + 2\log n +  2(n-1)\log M + \log 12 + \log\|r\|_\infty )\theta^{-\uni}\uni \label{e-neweps}
  \end{align}
  iterations.
  Let $x^*$ be the vector returned by the algorithm. Let us select pure policies $\sigma^*$ and $\tau^*$ reaching respectively the minimum and maximum in the expression of $T(x^*)$, as in~\eqref{e-selrule}. Then,  these policies are optimal. 
\end{theorem}
This theorem is proved in~\Cref{sec-proof-uni-turn}.

\section{Multiplicative Krasnoselskii-Mann Damping applied to Entropy Games}
\label{sec-entropy}
In the case of entropy games, the ergodic eigenproblem, for the operator $F$
defined in~\Cref{sec-entropy-games}, consists in finding $u\in\R^n$ and
$\lambda\in \R$ such that $\exp(\lambda) \exp(u)=  F(u)$.
Equivalently, $\lambda e +u = T(u)$ where $T=\log \circ F\circ \exp$.
\todo{BT : cette équation fait un min sur les $t$ mais aucun $t$ n'apparait dans ce qu'on minimise. Ce doit être un typo. SG. Merci basile, j'ai complete en precisant qu on minmax et somme sur les aretes}
If this equation is solvable, then $\exp(\lambda)$ is the value of the entropy game,
for all initial states $d\in \vertexsetD$.
To solve this equation,
we fix a positive number $\vartheta>0$, and consider the following ``multiplicative'' variant
of the Krasnoselskii-Mann operator:
\[
[\shapleymKM(v)]_d = \log\min_{t\in \vertexsetT, (d,t)\in\edges} \;\max_{p\in \vertexsetP,(t,p)\in\edges}   \big(\vartheta \exp(v_d) + \sum_{d'\in \vertexsetD,(p,d')\in\edges} m_{p,d'}\exp(v_{d'})\big)
\enspace .
\]
Unlike in the additive case, we do not perform a ``convex combination'' of the identity map and of the
Shapley operator, but we only add the ``diagonal term'' $\vartheta\exp(v_d)$, where $\vartheta$
can still interpreted as a ``damping intensity'', albeit in a multiplicative sense.
If $T(u) = \lambda e +u$, then, one readily checks that $\shapleymKM(u)= \mu e + u$,
where $\mu=\log(\vartheta+\exp(\lambda))$, and vice versa, so
the non-linear eigenproblems for $T$ and $\shapleymKM$ are equivalent.
As in the additive case, the damping intensity must be tuned
to optimize the complexity bounds. We shall say that the multiplicity
$m_{p,d'}$ is {\em off-diagonal} if there is no path $d'\to t\to p\to d'$ in
the graph of the game. Equivalently, for any choices of policies
$\sigma,\tau$ of the two players, the entry $m_{p,d'}$ does not
appear on the diagonal of the matrix $M^{\sigma,\tau}$, defined
in~\Cref{sec-entropy-games}.
Then, we denote by $\mbar$ the minimum of off-diagonal
multiplicities, observe that $\mbar$ is precisely the minimum
of all off-diagonal entries of the matrices $M^{\sigma,\tau}$
associated to all pairs of policies. We set $\vartheta\coloneqq \mbar$.

We shall say that an entropy game is {\em irreducible} if for every pair of policies
$\sigma,\tau$, the matrix $M^{\sigma,\tau}$ is irreducible.
The {\em irreducibility index} $\irr$ of an irreducible entropy game is the smallest integer
$k$ such that for all policies $\sigma_1,\tau_1,\dots,\sigma_k,\tau_k$, the matrix
$M^{\sigma_1\tau_1}\dots M^{\sigma_k\tau_k}$ has positive entries.
Arguing as in the case of stochastic concurrent games, we get that
$\irr\leq n$ as soon as the game is irreducible.
We define the {\em $l$-ambiguity} of the entropy game
\( \mathcal{A}_l \coloneqq \max_{d,d'\in\vertexsetD}\max_{\sigma_1,\tau_1,\dots,\sigma_l,\tau_l} (M^{\sigma_1\tau_1}\dots M^{\sigma_l\tau_l})_{d,d'}\).
Observe that \( (M^{\sigma_1\tau_1}\dots M^{\sigma_l\tau_l})_{d,d'}\) is the number of paths
from $d$ to $d'$ counted with multiplicities, in the finite horizon game induced
by the policies $\sigma_1,\tau_1,\dots,\sigma_l,\tau_l$
(this motivates the term ``$l$-ambiguity'').
If the game is irreducible,     
we define
the {\em ambiguity} of the game $\mathcal{A}\coloneqq \max_{1\leq l\leq \irr}\mathcal{A}_l^{1/l}$.
We set $W \coloneqq \max_{(p,d)\in \edges\cap (\vertexsetP\times\vertexsetD)} m_{p, d}$, and observe that\todo{SG: now bound
  $W\leq \mathcal{A}\leq n^{1-1/\irr}W$ instead of $n^{1-1/n}W$.}
$W\leq \mathcal{A}\leq n^{1-1/\irr}W$.
\begin{theorem}\label{km-mult}
  Let $\shapleymKM$ be the multiplicative Krasnoselskii-Mann operator of an irreducible entropy game. Then,
$\shapleymKM^{\irr}$ is a contraction in Hilbert's seminorm, with contraction rate bounded by
  $\frac{\mcst - 1}{\mcst + 1}$, where $\mcst \coloneqq  (1+\mathcal{A}/\mbar)^{\irr}$.
\end{theorem}
This result is proved in~\Cref{sec-proof-km-mult}.\todo{SG: complete appendix to deal with ambiguity: DONE}
We recall the following separation bound.
\begin{theorem}[Coro.\ of~\cite{icalp2022}]\label{th-sepentropy}
  Suppose two pairs of strategies yield distinct values in an entropy game with $n$ Despot's states. Then, these values differ at least by $\nu_{n}^{-1}$ where\todo{SG:I rewrote the bound in a clearer way}
  \[
  \nu_{n}\coloneqq
    2^{n} (n+1)^{8n}n^{2n^2+n+1} e^{4n^2}\max( 1,  W/ 2)^{4n^2} \enspace .
  \]
\end{theorem}
Then, using~\Cref{th-terminates}, we deduce:
\begin{theorem}\label{th-irr-entropy}
  Consider an irreducible entropy game, with irreducibility index $\irr$.
  Let us choose $\epsilon= (1+(\mbar+W)\nu_n)^{-1}$,\todo{SG: I changed the value of $\epsilon$ after optimizing the proof}
  so that Algorithm~\ref{RVI} applied to $\shapleymKM$ runs in at most\todo[color=red!30]{I changed the factor $(1+\mcst)/2$ to $\mcst/2$, see the proof} 
  \(
  (\log(1+(\mbar+W)\nu_n) + \log 6 )\irr \mcst/2
  \)
  iterations.
  Moreover, let $x^*$ be the vector returned by the algorithm. Let us select pure policies $\sigma^*$ and $\tau^*$ reaching respectively the minimum and maximum in the expression of $\shapleymKM(x^*)$. Then,  these policies are optimal. 
\end{theorem}

\section{Concluding Remarks}
We have established parameterized complexity bounds for relative value iteration applied to several
classes of stochastic games satisfying irreducibility conditions. These bounds rely on contraction properties in Hilbert's seminorm.
It would be interesting to see whether these contraction properties can also be exploited to derive
complexity bounds for policy iteration, instead of value iteration.
\bibliography{complement}

\def\cprime{$'$}
\begin{thebibliography}{10}

\bibitem{akian:hal-02143807}
M.~Akian, S.~Gaubert, J.~Grand-Cl{\'e}ment, and J.~Guillaud.
\newblock {The operator approach to entropy games}.
\newblock {\em {Theory of Computing Systems}}, 63:1089--1130, 2019.

\bibitem{ergodicity_conditions}
M.~Akian, S.~Gaubert, and A.~Hochart.
\newblock Ergodicity conditions for zero-sum games.
\newblock {\em Discrete Contin. Dyn. Syst.}, 35(9):3901--3931, 2015.

\bibitem{agn12}
M.~Akian, S.~Gaubert, and R.~Nussbaum.
\newblock Uniqueness of the fixed point of nonexpansive semidifferentiable
  maps.
\newblock {\em Trans. of AMS}, 368(2):1271--1320, February 2016.

\bibitem{Akian2001}
M.~Akian, A.~Sulem, and M.~I. Taksar.
\newblock Dynamic optimization of long-term growth rate for a portfolio with
  transaction costs and logarithmic utility.
\newblock {\em Mathematical Finance}, 11(2):153--188, April 2001.

\bibitem{icalp2022}
X.~Allamigeon, S.~Gaubert, R.~D. Katz, and M.~Skomra.
\newblock {Universal Complexity Bounds Based on Value Iteration and Application
  to Entropy Games}.
\newblock In Miko{\l}aj Boja\'{n}czyk, Emanuela Merelli, and David~P. Woodruff,
  editors, {\em 49th International Colloquium on Automata, Languages, and
  Programming (ICALP 2022)}, volume 229 of {\em Leibniz International
  Proceedings in Informatics (LIPIcs)}, pages 110:1--110:20, Dagstuhl, Germany,
  2022. Schloss Dagstuhl -- Leibniz-Zentrum f{\"u}r Informatik.

\bibitem{anantharam}
V.~Anantharam and V.~S. Borkar.
\newblock A variational formula for risk-sensitive reward.
\newblock {\em SIAM J. Contro. Optim.}, 55(2):961--988, 2017.
\newblock arXiv:1501.00676.

\bibitem{andersson_miltersen}
D.~Andersson and P.~B. Miltersen.
\newblock The complexity of solving stochastic games on graphs.
\newblock In {\em Proceedings of the 20th International Symposium on Algorithms
  and Computation (ISAAC)}, volume 5878 of {\em Lecture Notes in Comput. Sci.},
  pages 112--121. Springer, 2009.

\bibitem{asarin_entropy}
E.~Asarin, J.~Cervelle, A.~Degorre, C.~Dima, F.~Horn, and V.~Kozyakin.
\newblock Entropy games and matrix multiplication games.
\newblock In {\em Proceedings of the 33rd International Symposium on
  Theoretical Aspects of Computer Science (STACS)}, volume~47 of {\em LIPIcs.
  Leibniz Int. Proc. Inform.}, pages 11:1--11:14, Wadern, 2016. Schloss
  Dagstuhl--Leibniz-Zentrum f{\"u}r Informatik.

\bibitem{miguel}
L.~Attia and M.~Oliu-Barton.
\newblock A formula for the value of a stochastic game.
\newblock {\em PNAS}, 52(116):26435--26443, 2019.

\bibitem{baillonbruck}
J.~B. Baillon and R.~E. Bruck.
\newblock Optimal rates of asymptotic regularity for averaged nonexpansive
  mappings.
\newblock In K.~K. Tan, editor, {\em Proceedings of the Second International
  Conference on Fixed Point Theory and Applications}, pages 27--66. World
  Scientific Press, 1992.

\bibitem{bewley_kohlberg}
T.~Bewley and E.~Kohlberg.
\newblock The asymptotic theory of stochastic games.
\newblock {\em Math. Oper. Res.}, 1(3):197--208, 1976.

\bibitem{Boros2016}
E.~Boros, Kh. Elbassioni, V.~Gurvich, and K.~Makino.
\newblock A potential reduction algorithm for two-person zero-sum mean payoff
  stochastic games.
\newblock {\em Dynamic Games and Applications}, 8(1):22--41, July 2018.

\bibitem{DBLP:journals/corr/ChatterjeeI14}
K.~Chatterjee and R.~Ibsen{-}Jensen.
\newblock The complexity of ergodic mean-payoff games.
\newblock Extended version of a paper published in the proceedings of ICALP,
  2014.
\newblock \href {http://arxiv.org/abs/1404.5734} {\path{arXiv:1404.5734}}.

\bibitem{condon}
A.~Condon.
\newblock The complexity of stochastic games.
\newblock {\em Inform. and Comput.}, 96(2):203--224, 1992.

\bibitem{Dobrushin1956I}
R.~L. Dobrushin.
\newblock Central limit theorem for nonstationary {M}arkov chains. {I}.
\newblock {\em Theory of Probability {\&} Its Applications}, 1(1):65--80,
  January 1956.

\bibitem{Dobrushin1956II}
R.~L. Dobrushin.
\newblock Central limit theorem for nonstationary {M}arkov chains. {II}.
\newblock {\em Theory of Probability \& Its Applications}, 1(4):329--383,
  January 1956.

\bibitem{Etessami2008}
K.~Etessami and M.~Yannakakis.
\newblock Recursive concurrent stochastic games.
\newblock {\em Logical Methods in Computer Science}, 4(4), November 2008.

\bibitem{FEDERGRUEN1978711}
A~Federgruen, P.J Schweitzer, and H.C Tijms.
\newblock {Contraction mappings underlying undiscounted Markov decision
  problems}.
\newblock {\em Journal of Mathematical Analysis and Applications},
  65(3):711--730, 1978.

\bibitem{arxiv1}
S.~Gaubert and J.~Gunawardena.
\newblock The {P}erron-{F}robenius theorem for homogeneous, monotone functions.
\newblock {\em Trans. of AMS}, 356(12):4931--4950, 2004.

\bibitem{stott2020}
S.~Gaubert and N.~Stott.
\newblock A convergent hierarchy of non-linear eigenproblems to compute the
  joint spectral radius of nonnegative matrices.
\newblock {\em Mathematical Control and Related Fields}, 10(3):573--590, 2020.

\bibitem{gillette}
D.~Gillette.
\newblock {\em Stochastic games with zero stop probabilities}, volume~3.
\newblock Princeton University Press, 1957.

\bibitem{tsigaridas}
K.~Arnsfelt Hansen, M.~Koucky, N.~Lauritzen, P.~Bro Miltersen, and E.~P.
  Tsigaridas.
\newblock Exact algorithms for solving stochastic games.
\newblock In {\em STOC 2011}, 2011.

\bibitem{hoffman_karp}
A.~J. Hoffman and R.~M. Karp.
\newblock On nonterminating stochastic games.
\newblock {\em Manag. Sci.}, 12(5):359--370, 1966.

\bibitem{Howard-Matheson}
R.~A. Howard and J.~E. Matheson.
\newblock Risk-sensitive {M}arkov decision processes.
\newblock {\em Management Science}, 18(7):356--369, 1972.

\bibitem{ishikawa}
S.~Ishikawa.
\newblock Fixed points and iteration of a nonexpansive mapping in a {B}anach
  space.
\newblock {\em Proceedings of the American Mathematical Society}, 59(1):65--71,
  1976.

\bibitem{krasno}
M.~A. Krasnosel'ski\u{i}.
\newblock Two remarks on the method of successive approximations.
\newblock {\em Uspekhi Matematicheskikh Nauk}, 10:123--127, 1955.

\bibitem{nussbaumlemmens}
B.~Lemmens and R.~Nussbaum.
\newblock {\em Nonlinear Perron-Frobenius Theory}, volume 189 of {\em Cambridge
  Tracts in Mathematics}.
\newblock Cambridge University Press, May 2012.

\bibitem{LL69}
T.~M. Liggett and S.~A. Lippman.
\newblock Stochastic games with perfect information and time average payoff.
\newblock {\em SIAM Rev.}, 11:604--607, 1969.

\bibitem{mann}
W.~R. Mann.
\newblock Mean value methods in iteration.
\newblock {\em Proceedings of the American Mathematical Society}, 4:506--510,
  1953.

\bibitem{mertens_neyman}
J.-F. Mertens and A.~Neyman.
\newblock Stochastic games.
\newblock {\em Internat. J. Game Theory}, 10(2):53--66, 1981.

\bibitem{sorin_repeated_games}
J.-F. Mertens, S.~Sorin, and S.~Zamir.
\newblock {\em Repeated games}, volume~55 of {\em Econom. Soc. Monogr.}
\newblock Cambridge University Press, Cambridge, 2015.

\bibitem{mills}
H.D. Mills.
\newblock Marginal values of matrix games and linear programs.
\newblock In H.~W. Kuhn and A.~W. Tucker, editors, {\em Linear Inequalities and
  Related Systems}, volume~38 of {\em Annals of Mathematics Studies}, pages
  183--194. Princeton University Press, 1956.

\bibitem{Roc94}
R.T. Rockafellar and R.J.-B. Wets.
\newblock {\em Variational Analysis}.
\newblock Springer-Verlag, New York, 1997.

\bibitem{RS01}
D.~Rosenberg and S.~Sorin.
\newblock An operator approach to zero-sum repeated games.
\newblock {\em Israel J. Math.}, 121(1):221--246, 2001.

\bibitem{rothblum}
U.~G. Rothblum.
\newblock Multiplicative {M}arkov decision chains.
\newblock {\em Mathematics of Operations Research}, 9(1):6--24, 1984.

\bibitem{rothblumwhittle}
U.~G. Rothblum and P.~Whittle.
\newblock Growth optimality for branching {M}arkov decision chains.
\newblock {\em Mathematics of Operations Research}, 7(4):582--601, 1982.

\bibitem{shapley_stochastic}
L.~S. Shapley.
\newblock Stochastic games.
\newblock {\em Proc. Natl. Acad. Sci. USA}, 39(10):1095--1100, 1953.

\bibitem{skomra_bounds}
M.~Skomra.
\newblock Optimal bounds for bit-sizes of stationary distributions in finite
  {M}arkov chains.
\newblock 2021.

\bibitem{Sladky1976}
K.~Sladk{\'y}.
\newblock {\em On dynamic programming recursions for multiplicative Markov
  decision chains}, pages 216--226.
\newblock Springer Berlin Heidelberg, Berlin, Heidelberg, 1976.
\newblock \href {https://doi.org/10.1007/BFb0120753}
  {\path{doi:10.1007/BFb0120753}}.

\bibitem{tsitsiklis}
J.~N. Tsitsiklis.
\newblock {NP}-hardness of checking the unichain condition in average cost
  {MDP}s.
\newblock {\em Oper. Res. Lett.}, 35(3):319--323, 2007.

\bibitem{Vigeral2013}
G.~Vigeral.
\newblock A zero-sum stochastic game with compact action sets and no asymptotic
  value.
\newblock {\em Dynamic Games and Applications}, 3(2):172--186, January 2013.
\newblock \href {https://doi.org/10.1007/s13235-013-0073-z}
  {\path{doi:10.1007/s13235-013-0073-z}}.

\bibitem{WHITE1963373}
D.J White.
\newblock Dynamic programming, {M}arkov chains, and the method of successive
  approximations.
\newblock {\em Journal of Mathematical Analysis and Applications},
  6(3):373--376, 1963.

\bibitem{zijmjota}
W.~H.~M. Zijm.
\newblock Asymptotic expansions for dynamic programming recursions with general
  nonnegative matrices.
\newblock {\em J. Optim. Theory Appl.}, 54(1):157--191, 1987.
\newblock \href {https://doi.org/10.1007/BF00940410}
  {\path{doi:10.1007/BF00940410}}.

\bibitem{zwick_paterson}
U.~Zwick and M.~Paterson.
\newblock The complexity of mean payoff games on graphs.
\newblock {\em Theoret. Comput. Sci.}, 158(1--2):343--359, 1996.
\newblock \href {https://doi.org/10.1016/0304-3975(95)00188-3}
  {\path{doi:10.1016/0304-3975(95)00188-3}}.

\end{thebibliography}

\appendix
\section{Example of turn-based stochastic mean-payoff game}
\label{sec-example}
\begin{figure}[htbp]
\caption{Example of a turn-based stochastic mean-payoff game. Min states are represented by squares; Max states are represented by circles; Nature states are represented by small diamonds. The payments made by Min to Max are shown on the arcs. For every Nature state, the next state is chosen with the uniform distribution among the successors. Optimal policies of Min and Max are shown in bold (red and blue arcs, respectively).}
 
\begin{center}
  \includegraphics{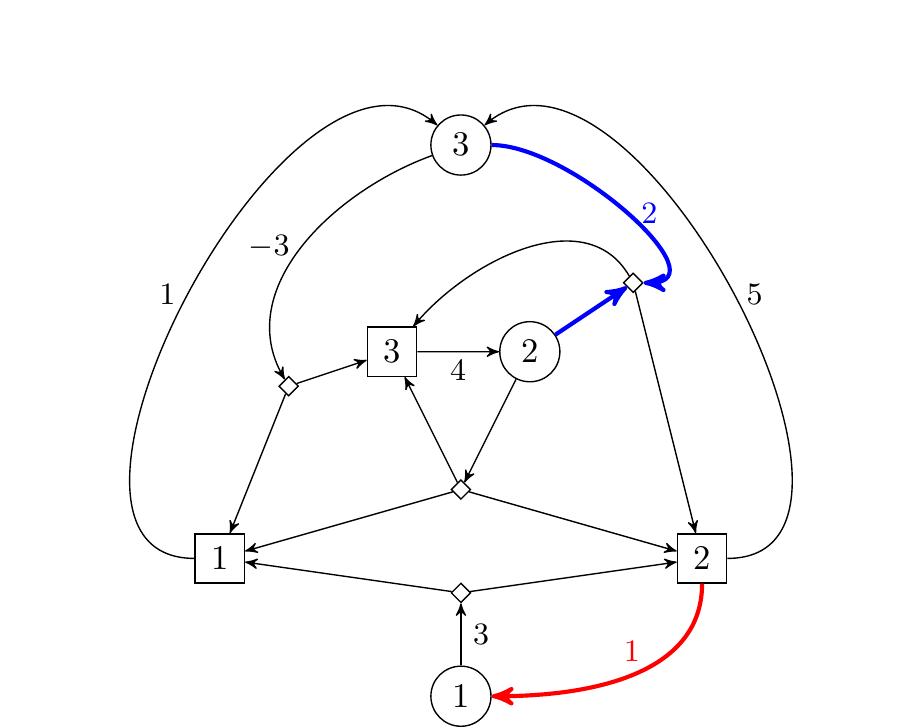}
\end{center}\label{ex-turnbased}
\end{figure}

An example of turn-based stochastic mean-payoff game is represented in~\Cref{ex-turnbased}.\todo[color=red!30]{SG: I chaned the convention on the figure, payments made by Min to Max for all arcs, check signs!}
The Shapley operator of this game is the map $T: \R^3\to \R^3$ given by
     \begin{align*}
    T_1(x) &=  1+ \max\big(2+ \frac{x_2+x_3}{2}, -3 + \frac{x_1+x_3}{2}\big),\\
    T_2(x) &= \min\Big( 5 + \max\big(2+ \frac{x_2+x_3}{2}, -3 + \frac{x_1+x_3}{2}\big),
    1+ 3 + \frac{x_1+x_2}{2}\Big)\enspace ,\\
    T_3(x) &=  4+ \max(\frac{x_2+x_3}{2}, \frac{x_1 + x_2 + x_3}{3}) \enspace .
     \end{align*}
The unichain index defined in~\Cref{sec-unichain} is $\uni=1$. Indeed, for all pairs\todo[color=blue!30]{SG: I explained the notation. I also added the explicit bound on the contraction rate, and explicit computation of the unichain index, check!}
     of policies $(\sigma_1,\tau_1)$, we have $S_1(\sigma_1,\tau_1)\supset \{1,3\}$, $S_3(\sigma_1,\tau_1)\supset \{2,3\}$,
     and $S_2(\sigma_1,\tau_1)\supset\{2,3\}$ if $\sigma_1$ sends Min state $2$
     to Max state $3$, and $S_2(\sigma_1,\tau_1)= \{1,2\}$ if $\sigma_1$ sends
     Min state $2$ to Max state $1$. In all cases, we have $S_i(\sigma_1,\tau_1)
     \cap S_j(\sigma_1,\tau_1)\neq\varnothing$ for $i\neq j$.
   We have $\pmin=1/3$, and $\theta=
     \pmin/(1+\pmin)=1/4$. It follows from~\Cref{contraction_concu} that the damped Shapley operator
     $T_\theta$ is a contraction of rate $3/4$.
     We know from~\Cref{th-exists-unichain} that the ergodic eigenproblem is solvable.
     By applying~\Cref{RVI}, we find $T(u)= \lambda e + u$ with $(-1,-0.5,0)$ and $\lambda = 3.75$.
     An approximation of $u$ of precision $<10^{-8}$ in the sup norm is reached after only $15$ iterations, to be compared with the precision
     of order $(3/4)^{15}\simeq 10^{-2}$ given by the theoretical upper bound, for the same number of iterations.
     Thus, the convergence may be faster in practice than the one shown
     in \Cref{cor-complexity-concur}.
     We deduce from $T(u)= \lambda e + u$ that the value of the mean-payoff game is $3.75$ regardless
     of the initial state. Optimal policies $\sigma$ and $\tau$ of both players are obtained by selecting the actions
     that achieve the minimum or the maximum in the expression of $T(u)$. The non-trivial
     actions of these optimal policies are as follows:
     from Min state $2$ (square at bottom right), go to Max state 3 (circle at the top level),
     from Max state 2 (circle at the middle level), and also from Max state 3, got to the top right state (diamond) of Nature. These actions are shown on~\Cref{ex-turnbased}. The stochastic matrix $P^{\sigma,\tau}$
     and payment vector $r^{\sigma,\tau}$ associated to these
     policies are given by
     \[
     r^{\sigma,\tau} = \left(\begin{array}{c}
       3
       \\
       4
       \\
       4
       \end{array}\right),\qquad 
     P^{\sigma,\tau}= \left(\begin{array}{ccc}
       0 & 1/2& 1/2\\
       1/2& 1/2& 0\\
       0 & 1/2& 1/2
       \end{array}\right) \enspace .
     \]
     The unique invariant measure of the matrix $P^{\sigma,\tau}$
     is $\pi=(1/4,1/2,1/4)$, and we have $\pi r^{\sigma,\tau}=15/4=3.75$,
     consistently with the value of the mean-payoff already found.

\section{Proof of~\Cref{obs-contract}}\label{sec-proof-obs}
Let $[x]\coloneqq x+\R e $ denote the equivalence
class of $x$ in $\R^n/\R e$.
Then, the induced map $\overline{T}:\R^n/\R e \to \R^n /\R e $,
$[x]\mapsto [T(x)]$, is 
such that $\overline{T}^q$ is a $\gamma$-contraction in $\|\cdot\|_H$.
Moreover, 
$\R^n/\R e$ equipped with $\|\cdot\|_H$ is a Banach space.
Hence, $\overline{T}^q$ has a unique fixed point $[u]$. Moreover,
$\overline{T}^{q}(\overline{T}([u])) = \overline{T}(\overline{T}^q([u]))=\overline{T}([u])$,
and by uniqueness of the fixed point of $\overline{T}^q$, we deduce
that $\overline{T}([u])=[u]$. This means precisely that $T(u)=\lambda e + u$
for some $\lambda \in \R$.

\section{Derivation of~\Cref{th-dob}}\label{sec-explain-dobrushin}
  The original statement of Dobrushin,
  given in~\cite[Section 1.4, (1.5'') and (1.19)]{Dobrushin1956I},
  and proved in~\cite{Dobrushin1956II}, shows that
  $\delta(M)$ is the operator norm of $M^T$ acting on the space
  $H=\{x\in \R^n\mid e\cdot x =0\}$, equipped with the $\ell_1$-norm
  $\|\cdot\|_1$.
  The space $(\R^n/\R e,\hilbert{\cdot})$ is the dual of $(H,\|\cdot\|_{1})$.
  Since an operator between two Banach spaces and its adjoint, acting on the dual
  spaces, have the same operator norms, it follows
  that $\hilbert{M}= \delta(M)$.

\section{Complements on stochastic zero-sum games}

For any pair $(\alpha,\beta)\in \randpolpair$ of randomized policies, we define the  vector $r^{\alpha,\beta}=(r_i^{\alpha_i,\beta_i})_{i\in [n]}\in \R^n$ and
 the stochastic matrix $P^{\alpha,\beta}=(P^{\alpha_i,\beta_i}_{ij})_{i,j\in [n]}$, according to~\eqref{e-int}.
Let us also define, the maps $T^{\alpha,\beta}$, $T^{\alpha}$ and 
$\Tbeta$ from $\R^n$ to itself, such that, for all $v\in \R^n$, we have
\begin{align*}
T^{\alpha,\beta} (v)&= r^{\alpha,\beta}+P^{\alpha,\beta} v\\
T^{\alpha} (v)_i &=    \sup_{\beta_i \in \Delta_{B(i)} }
\Big( r_i^{\alpha_i,\beta_i} + \sum_{j\in [n]} P_{ij}^{\alpha_i,\beta_i} v_j\Big)
\;\text{ for }i\in [n]\enspace,\\
\Tbeta (v) &= 
\inf_{\alpha_i \in \Delta_{A(i)} } 
\Big( r_i^{\alpha_i,\beta_i} + \sum_{j\in [n]} P_{ij}^{\alpha_i,\beta_i} v_j\Big)
\;\text{ for } i\in [n] 
\enspace .
\end{align*}
Then, for all $v\in\R^n$, there exists $(\alpha,\beta)\in \randpolpair$ such that $T(v)=T^{\alpha}(v)=\Tbeta(v)=T^{\alpha,\beta}(v)$.
Such policies are called optimal policies with respect to $v$,
and we denote their sets as follows:
\begin{subequations}\label{defopt}
\begin{gather}
\Optmin(v)= \{ \alpha\in \randpolMin\mid  T(v) = T^{\alpha}(v) \}\\
\Optmax(v)= \{ \beta\in\randpolMax \mid T(x) = \Tbeta(v) \}\enspace .
\end{gather}
\end{subequations}
\begin{remark}
Applying the above notations to pure policies $\alpha, \beta$ identified 
to the elements $\sigma\in \prod_{i\in [n]}A(i)$ and $\tau\in \prod_{i\in [n]} B(i)$ as in \Cref{repeated_games}, we get that for all $(\sigma,\tau)\in \purepolpair$, and all $i,j\in [n]$, we have 
\[ r^{\sigma,\tau}_i=r_i^{\sigma(i),\tau(i)}\quad \text{ and } \quad
P_{i,j}^{\sigma,\tau} = P_{i,j}^{\sigma(i),\tau(i)}\enspace .\]

Every pair $(\alpha,\beta)\in \randpolpair$ of randomized policies yields
a product probability measure $\nu^{\alpha,\beta}$ on the space of pairs
of pure policies $(\sigma,\tau)\in \purepolpair$,
given by $\nu^{\alpha,\beta} = (\alpha_1\otimes \beta_1)\otimes \dots\otimes (\alpha_n\otimes \beta_n)$.
In other words, the actions played by the two players,
$\sigma(1),\tau(1),\dots,\sigma(n),\tau(n)$ are thought of as independent 
random variables with probability distribions $\alpha_1,\beta_1,\dots, \alpha_n,\beta_n$.
Then, the following identity is a consequence of the independence:
\begin{align} %
r^{\alpha,\beta} = \int_{\purepolpair}
r^{\sigma,\tau} d\nu^{\alpha,\beta}(\sigma,\tau)
\qquad\text{and}\qquad 
P^{\alpha,\beta} =
\int_{\purepolpair} 
P^{\sigma,\tau}
d\nu^{\alpha,\beta}(\sigma,\tau)
\enspace ,
\label{e-integral}
\end{align}
which implies that the payment vector and transition matrix associated to a pair $(\alpha,\beta)$
of randomized policies is a convex combination of the ones associated to 
pairs of pure policies.
\todo{SG: added the following remark to incorporate an observation of BT in the polyhedral case
  (what it means in terms of extreme points)}
\begin{remark}
  By Milman's converse to the Krein-Milman theorem, the second equality in~\eqref{e-integral}
entails that the set of extreme points of the compact convex set of matrices $\{P^{\alpha,\beta}\mid
(\alpha,\beta)\in\randpolpair\}$ is included in the set of transition matrices
associated with pure policies, $\{P^{\sigma,\tau}\mid
(\sigma,\tau)\in\purepolpair\}$ .
\end{remark}

\end{remark}

\section{Proof of~\Cref{th-exists-unichain}}\label{sec-proof-th-exists-unichain}
  Let $\hat{T}(x)\coloneqq \lim_{s\to\infty} s^{-1}T(sx)$ denote the {\em recession function} of the shapley operator $T$. 
  We rely on Theorem 3.1 of \cite{ergodicity_conditions} (using $g=0$ in $(iii)$),
  which entails that the ergodic eigenproblem for $T$ is solvable iff
  $\hat{T}$ has only trivial fixed points, meaning
  that $\hat{T}(\eta)=\eta$ with $\eta\in \R^n$ implies
  that $\eta\in \R e$. We deduce from~\eqref{e-def-shapley-concurrent} that\todo{SG: I gave detailes and completed the proof}
  \[
    [  \hat{T}(\eta)]_i  =
        \min_{\alpha_i \in \Delta_{A(i)} } \max_{\beta_i \in \Delta_{B(i)} }
\sum_{j\in [n]} P_{ij}^{\alpha_i,\beta_i} \eta_j \enspace .
  \]
  Since the action spaces are compact, and the transition
  probabilities depend continuously on the actions,
  we have a selection property,
  which entails that, for all $\eta\in\R^n$,
  then there exists randomized policies $\alpha,\beta$ such that $\hat{T}(\eta)= P^{\alpha,\beta} \eta$. Hence, if $\hat{T}(\eta)=\eta$, $\eta=P^{\alpha,\beta}\eta$ is an harmonic vector of $P^{\alpha,\beta}$. Since $P^{\alpha,\beta}$ is unichain by assumption, the harmonic vector $\eta$ must be constant. Hence, by Theorem~3.1 of~\cite{ergodicity_conditions}, there exists $u\in \R^n$ and $\lambda\in \R$ such that $T(u)=\lambda e +u$. Then we select randomized policies
  which attain the minimum and the maximum in the expression of $T(u)$
  (see~\eqref{e-def-shapley-concurrent}). These policies are optimal
  in the stochastic mean-payoff game.

\section{Proof of~\Cref{th-terminates}}\label{sec-proof-RVI}
The proof of~\Cref{th-terminates} relies on two elementary lemmas.
We set $R(x)\coloneqq x-\mytop(x)$.
\begin{lemma}\label{lem-elt}
  We have $\|R(x)-R(y)\|_H = \|x-y\|_H$ for all $x,y\in\R^n$.
  Moreover, $\|(R\circ \shapleyapp)^\ell(x)-(R\circ T)^\ell(x)\|_H\leq 2\ell\eta$
for all $x\in \R^n$ and $\ell\geq 1$. 
\end{lemma}
\begin{proof}
  The first property is obvious. We prove the inequality
  $\|(R\circ \shapleyapp)^\ell(x)-(R\circ T)^\ell(x)\|_H\leq \ell\eta$
  by induction on $\ell$. For $\ell=1$, this follows from the nonexpansiveness
  of $R$ in Hilbert seminorm, and from $\|\shapleyapp(x)-T(x)\|_H\leq 2
  \|\shapleyapp(x)-T(x)\|_\infty \leq 2\eta$. 
  Suppose now that $\|(R\circ \shapleyapp)^\ell(x)-(R\circ T)^\ell(x)\|_H\leq 2\ell\eta$.
  Then, using the nonexpansiveness
  of $R$ and $T$, we get $\|(R\circ \shapleyapp)^{\ell+1}(x)-(R\circ T)^{\ell+1}(x)\|_H\leq
  \|(R\circ \shapleyapp)^{\ell+1}(x)-(R\circ T)\circ (R\circ \shapleyapp^{\ell}(x)\|_H
  + \|(R\circ T)\circ (R\circ \shapleyapp)^{\ell}(x) - (R\circ T)^{\ell+1}(x)\|_H
\leq 2\eta + 2\ell\eta=2(\ell+1)\eta$.
\end{proof}
\begin{lemma}\label{lem-contractx}
  suppose that $T: \R^n\to \R^n$ is nonexpansive in Hilbert's seminorm,
  and that $T^q$ is a $\gamma$-contraction.
  Let $x_{k}\coloneqq T^k(x_0)$ with $x_0\in \R^n$.
  Then, $\hilbert{x_{k+1}-x_k}\leq \gamma^{\lfloor k/q\rfloor} \hilbert{T(x_0)-x_0}$.
\end{lemma}
\begin{proof}
  For all $k\geq 0$ and $r\geq 0$, we have $\hilbert{T^{kq+r+1}(x_0)-T^{qk+r}(x_0)}
    \leq \gamma^k \hilbert{T^{r+1}(x_0)-T^r(x_0)}\leq \hilbert{T(x_0)-x_0}$.
  \end{proof}
\begin{proof}[Proof of~\Cref{th-terminates}]
  Suppose that the algorithm terminates. We have
  \begin{align}
    T(x) \leq \eta e + \shapleyapp(x)  \leq x + (\eta +\beta) e \enspace, \label{e-ub}
    \end{align}
  and similarly, $T(x) \geq x + (\alpha -\eta) e$. Since
  $\eta\leq \epsilon/3$, using~\Cref{lemme_encadrement}, this provides the announced
  bound on the value of all initial states. Moreover, since
  $\beta-\alpha = \|x-\shapleyapp(x)\|_H\leq \epsilon/3$,
the interval
$[\alpha-\epsilon/3, \beta+\epsilon/3]$ is of width
$\beta-\alpha+2\epsilon/3\leq \epsilon$. This shows
the correctness of the algorithm.

Observe that $(R\circ T)^q(x)$ differs from $T^q(x)$
only by an additive constant. It follows that $(R\circ T)^q$
is also a $\gamma$-contraction in Hilbert seminorm.
By~\Cref{lem-contractx}, 
this entails that the sequence $x_k\coloneqq (R\circ T)^k(0)$
satisfies
\begin{align}
  \|x_{k+1}-x_k\|_H \leq \gamma^{\lfloor k/q\rfloor}\|T(0)\|_H \enspace .\label{e-contract}
\end{align}
Let us compare this sequence $x_k$
with the sequence $\xapprox_0=0$, $\xapprox_{k+1}= R\circ \shapleyapp(\xapprox_k)$
constructed by~\Cref{RVI},
using the inexact oracle $\shapleyapp$. We first prove by induction
  on $k\geq 0$ that
  \begin{align}
    \|\xapprox_{qk}-x_{qk}\|_H \leq 2q\eta/(1-\gamma)
    \label{e-bound}
    \end{align}
  For $k=0$, the property is trivial. If $\|\xapprox_{qk}-x_{qk}\|_H \leq 2q\eta/(1-\gamma)$, then, using~\Cref{lem-elt} with $\ell=q$, we get
  $\|\xapprox_{qk+q}-x_{qk+q}\|_H =
  \|(R\circ \shapleyapp)^q(\xapprox_{qk})
  - (R\circ {T})^q({x}_{qk})
  \|_H
  \leq 
  \|(R\circ \shapleyapp)^q(\xapprox_{qk})
  - (R\circ {T})^q(\xapprox_{qk})
  \|_H
  + 
  \|(R\circ T)^q(\xapprox_{qk})
  - (R\circ T)^q({x}_{qk})
  \|_H
  \leq
  2\eta q + 
  \gamma 2\eta q/(1-\gamma)
  = 2\eta q/(1-\gamma)
  $,
  because $(R\circ T)^q$
is a $\gamma$-contraction in Hilbert seminorm,
  which shows~\eqref{e-bound}.

  Moreover, using the nonexpansiveness of $T$, and using~\eqref{e-bound},
  \begin{align*}
    \hilbert{\xapprox_{qk}-\shapleyapp(\xapprox_{qk})}
  &\leq \hilbert{\xapprox_{qk}-x_{qk}}
    +\hilbert{x_{qk}-T(x_{qk})}
    \\
    &  
    +\hilbert{T(x_{qk})-T(\xapprox_{qk})}
    +\hilbert{T(\xapprox_{qk})-\tilde T(\xapprox_{qk})}
  \\
  &\leq \frac{2\eta q}{1-\gamma}
  + \gamma^k\hilbert{T(0)}  
  +   \frac{2\eta q}{1-\gamma}
  + 2\eta \enspace \\
  & = \eta (2 + \frac{4 q}{1-\gamma})+ \gamma^k \hilbert{T(0)} \enspace .
  \end{align*}
Hence, \Cref{RVI} terminates as soon as the condition
$\eta (2 + \frac{4 q}{1-\gamma})+ \gamma^k \|T(0)\|_H \leq \epsilon/3$
is fulfilled.
If the assumption of the theorem, $\eta(12 + \frac{24 q}{1-\gamma})\leq \epsilon$, holds, then the previous condition is fulfilled when $\epsilon/6 + \gamma^k \|T(0)\|_H \leq \epsilon/3$, i.e., when
$k\geq (\log\|T(0)\|_H+ \log 6 +|\log\epsilon| )/|\log \gamma|$.
  \end{proof}

\section{Proof of \Cref{th-stott}}
\label{appendix-proofofcvg} 
    Suppose that $T(u) = \lambda e + u$. Then $\shapleytheta(u) = (1-\theta) {\lambda}e + u$.
    Moreover, $\shapleytheta$ commutes with the additive action of constants, and is non-expansive for $\hilbert{\cdot}$. Therefore, $\shapleytheta$ induces a quotient operator $\Tilde{\shapleytheta}: \R^n/\R e \rightarrow \R^n/\R e$. Also, $(\R^n/\R e, \norm{\cdot}_{\mathrm{H}})$ inherits a finite-dimensional normed vector space structure, on which $\Tilde{\shapleytheta}$ is non-expansive.

    Let $\bar{u} = \{u + \alpha e, \alpha \in \R \}$ be the equivalence class of $u$. Then $\bar{u}$ is fixed point of the quotient operator.
    Let us note $\{x_k, k \in \NN \}$ the sequence produced by applying the \Cref{RVI} to $\shapleytheta$. According to the Ishikawa theorem, $\{\bar{x}_k, k \in \NN \}$ converges in $\R^n/\R e$.
    As $\mathbf{t}(x_k)=0$, this implies that $x^k$ converges in $\R^n$.
    
    Conversely, suppose that the sequence $\{x_k, k \in \NN \}$ converges in $\R^n$, let us note $u$ its limit. Then $\norm{T(u) - u}_{\mathrm{H}} = 0$, so $T(u) - u \in \R e$, which implies that there exists $\lambda\in\R$ such that $T(u) = \lambda e + u$.\hfill\qed
\section{Proof of~\Cref{thm_unichain} and~\Cref{prop-irr}}
\label{sec-proof-unichain}
\begin{proof}[Proof of~\Cref{thm_unichain}]
    Consider a sequence of pairs of pure policies $\sigma_1, \tau_1, \ldots , \sigma_k, \tau_k,\dots$.
    For all $i\in [n]$, define, for simplicity of notation, $S_{i}(k)\coloneqq
    S_i(\sigma_1, \tau_1, \ldots , \sigma_k, \tau_k)=  \{j\mid  [Q^{ \sigma_1, \tau_1} \dots Q^{ \sigma_k, \tau_k}]_{i j}>0\}$. 
    We have $S_i(k)\subset S_i(k+1)$
    because the diagonal coefficients of all the matrices $Q^{\sigma,\tau}$
    are strictly positive.

    Let us know consider two distinct elements $i_1,i_2$ of $[n]$.
    Assume by contradiction that $S_{i_1}(n) \cap S_{i_2}(n) = \varnothing$. Then for all $k < n$, $S_{i_1}(k) \cap S_{i_2}(k) = \varnothing$.
    Let us consider the sequence of subsets of $[n]$ consisting of $S_{i_1}(1) \cup S_{i_2}(1)$, $S_{i_1}(2) \cup S_{i_2}(2), \cdots$, and $S_{i_1}(n) \cup S_{i_2}(n)$. Since this is a nonincreasing sequence, and since the cardinality of $S_{i_1}(1) \cup S_{i_2}(1)$ is greater than $2$, there is an integer $m < n$ such that: $S_{i_1}(m-1) = S_{i_1}(m)$ and $S_{i_2}(m-1) = S_{i_2}(m)$.

    Hence $S_{i_1}(m-1)$ and $S_{i_2}(m-1)$ are closed under the action of the matrix
    $Q^{\sigma_m,\tau_m}$. Since this matrix is unichain, every set
    that is closed under its action must contain the final
    class of the matrix. It follows that $S_{i_1}(m-1)\cap S_{i_2}(m-1)$
    contains this final class, contradicting $S_{i_1}(m-1)\cap S_{i_2}(m-1)=\varnothing$.

    Let $k \leq n$ denote the minimal integer $k$ such that
    for all $i_1,i_2$, $S_{i_1}(k)\cap S_{i_2}(k)\neq \varnothing$,
    and let $Q\coloneqq  Q^{ \sigma_1, \tau_1} \dots Q^{ \sigma_k, \tau_k}$. 
Then, $\underset{i<j}{\min} \sum_l \min(Q_{i,l}, Q_{j,l})
    \geq \underset{i<j}{\min} \sum_{l\in F} \min(Q_{i,l}, Q_{j,l})> 0$, so $\underset{i<j}{\min} \sum_l \min(Q_{i,l}, Q_{j,l}) \geq  \theta^k$, and, using~\Cref{th-dob},
    \begin{equation*}
        \norm{Q^{ \sigma_1, \tau_1} \dots  Q^{ \sigma_k, \tau_k}}_{\mathrm{H}} = 1 - \underset{i<j}{\min} \sum_l \min(Q_{i,l}, Q_{j,l}) \leq 1 - \theta^k \enspace .
    \end{equation*}
\end{proof}
\begin{proof}[Proof of~\Cref{prop-irr}]
  The only set invariant under the action of an irreducible $n\times n$ stochastic matrix is $[n]$. Hence, the proof of~\Cref{thm_unichain} implies in this case that there is an index $k\leq n$
  such that the matrix $Q^{ \sigma_1, \tau_1} \dots Q^{ \sigma_k, \tau_k}$ is positive. Since all the entries 
  of this matrix are bounded from below by $\theta^k$, applying~\Cref{th-dob}, we deduce~\Cref{prop-irr}.
\end{proof}

\begin{remark}%
  The proof of \Cref{thm_unichain} builds on the idea of the proof of $(4)\Rightarrow (S)$ of Theorem 4 of \cite{FEDERGRUEN1978711}, correcting a gap there. In fact, the authors
  of~\cite{FEDERGRUEN1978711} consider, more generally, two different sequences of policies. Denote as above, $S_{i_1}(k)$ the set of states accessible from $i_{1}$ in at most $k$ steps, when applying the first sequence
  of policy, and $W_{i_2}(k)$ the analogous set, starting from $i_2$, and applying the second sequence of policies. It is shown in~\cite{FEDERGRUEN1978711} that
  there is an integer $k$ such that $S_{i_1}(k)$ and $W_{i_2}(k)$ are {\em closed}
  sets, where closed means what we call invariant under the action of a stochastic
  matrix. However, they are closed sets under the action of {\em different}
  stochastic matrices, and then we are unable to conclude that $S_{i_1}(k)\cap W_{i_2}(k)\neq\varnothing$ as claimed at the last line of proof of Theorem~4 of~\cite{FEDERGRUEN1978711}.
  However, we will use here
  a new approach to compute the contraction rate of the Shapley operator, exploiting
  properties of variational analysis. To apply our approach, we need only to consider
  one sequence instead of two, comparing $S_{i_1}(k)$ and $S_{i_2}(k)$, and
  in this special case, the above argument is valid.
\end{remark}

\begin{remark}
  Tsitsiklis showed in~\cite{tsitsiklis}
  that in the one player case,
  checking whether all policies are unichain is co-NP-complete.
\end{remark}
\section{Proof of \Cref{contraction_concu}}\label{proof-km-concur}
Recall that the (one-sided) {\em directional derivative} of a map $T:\R^n \to \R^p$,
at a point $x\in \R^n$, is the map $T'_x: \R^n\to \R^p$, defined by
\begin{align}
T'_x(h)\coloneqq \lim_{s\to 0^+}s^{-1}(T(x+sh)-T(x)) \enspace .\label{e-unidir}
\end{align}
The following is an extension of a classical result of Mills~\cite{mills}, which
follows from~\cite{RS01}.
\begin{proposition}[See~{\cite[Prop.~4]{RS01}}]\label{selection_derive}
    The Shapley operator of a concurrent game does admit a directional derivative,
    given by:
    \[ T'_x(h) = \sup_{\alpha \in \Optmax(v)} \inf_{\beta \in \Optmin(v)}
       P^{\alpha,\beta}h,\qquad \forall h \in \mathbb{R}^{n}, \forall x\in \R^n \enspace.
       \]
       Moreover, for all $x,h$, there exists randomized policies $\alpha$ and $\beta$
such that $T'_x(h)= P^{\alpha,\beta}h$. 
\end{proposition}
However, we will need a stronger notion than the mere existence of directional derivative.
A map $T$ is said to be {\em semidifferentiable} at point $x$ if the limit in~\eqref{e-unidir} is {\em uniform} in $h$, when $h$ ranges over the unit sphere. Equivalently, one requires
the existence of a first-order ``Taylor-like'' extension $T(x+h)=T(x)+G_x(h) + o(\|h\|)$
where $G_x$ is a map that is continuous and positively homogeneous of degree $1$. Then,
$G_x=T'_x(h)$. In other words, semidifferentiability is akin to differentiability,
but the semidifferential map $T'_x$ is allowed to be non-linear. We refer
the reader to~\cite{Roc94} for background on this notion.
If $T$ is Lipschitz continuous, and if $T$ admits directional derivatives
in all directions at a point $x$, then, it is semidifferentiable at this point,
see e.g.~\cite[Lemma~3.2]{agn12}. It follows from~\Cref{selection_derive} 
that the Shapley operator of a concurrent game is semidifferentiable at every point. There is also
a chain rule for semidifferentiable maps. If $F: \R^n\to \R^p$ it semidifferentiable
at a point $x\in \R^n$, and $G:\R^p\to \R^q$ is semidifferentiable at point $F(x)$,
then, $G\circ F$ is semidifferentiable at point $x$, with
\begin{align}
  (G\circ F)'_x (h) = G'_{F(x)}\circ F'_x (h) \enspace .
  \label{e-chain}
\end{align}
see e.g.~\cite[Lemma~3.4]{agn12}.

Let us now denote, for $x,y \in \mathbb{R}^n$, 
\[ f : t \rightarrow T^k(x+t(y-x))
\]
By Rademacher's theorem\todo{SG: big gun, more elementary argument?},
a Lipschitz continuous map $\R\to \R^n$ is almost everywhere
differentiable, and it is absolutely continuous. It follows that
\begin{align}
  T^k(y) - T^k(x) = \int_0^{1} f'(s)ds\enspace ,\label{intscalar}
\end{align}
where $f'(s)$ denote the derivative of $f$ at point $s$, almost everywhere defined.
When it exists, this derivative coincides with $f'_s(1)$, the semidifferential of the map $f$
at point $s$, in the direction $1\in \R$. Then, we arrive at the following result.
\begin{lemma}\label{lem-selec}
  For all $s\in (0,1)$, there exists randomized policies
  $(\alpha_s^1,\beta_s^1),\dots ,(\alpha_s^k,\beta_s^k) \in \randpolpair$ such that:
\begin{equation} \label{selection_f_k_eq}
    f'(s) = P^{\alpha_s^1,\beta_s^1} \dots P^{\alpha_s^k,\beta_s^k} (y-x) \enspace .
\end{equation}
\end{lemma}
\begin{proof}
 Using the chain rule~\eqref{e-chain}, we get that for all $s\in (0,1)$,
\[ f'(s) = T'_{T^{k-1}(x+s(y-x))} \circ \cdots \circ  T'_{x+s(y-x)} (y-x)
\]
Then, the conclusion follows from the selection result in~\Cref{selection_derive}.
\end{proof}

Let us consider the following quantity
\[D_k \coloneqq \max_{(\alpha_1,\beta_1),\dots  ,(\alpha_k,\beta_k) \in \randpolpair}
\norm{P^{\alpha_1\beta_1}\dots P^{\alpha_k,\beta_k}}_{\mathrm{H}} \enspace .\]
Observe that the maximum is taken over a Cartesian product of compact sets (of randomized strategies),
and that the expression which is maximized is continuous in
$(\alpha_1,\beta_1),\dots  ,(\alpha_k,\beta_k)$. Hence,
the maximum is achieved.\todo{SG: not a polytope anymore}

\begin{lemma}\label{D_k} 
For all $x,y \in \RR^n$ and $k \in \NN$
\[ \norm{T^k(x) - T^k(y)}_{\mathrm{H}} 
\leq D_k \norm{y-x}_{\mathrm{H}}
\]
\end{lemma}
\begin{proof}
We deduce from~\Cref{lem-selec} and~\eqref{intscalar} that 
  \[
  \norm{T^k(x) - T^k(y)}_{\mathrm{H}}
  \leq \int_0^1 
\norm{f'(s)}_{\mathrm{H}}ds  \leq D_k \norm{y-x}_{\mathrm{H}} \enspace .
\]
  \end{proof}

The following lemma shows that the maximum in the expression of $D_k$ is actually achieved
by sequences of {\em pure} policies.\todo{SG: rewrote the proof in a shorter more abstract way}
\begin{lemma}\label{lemme_barycentre}
  Let $((\alpha_1,\beta_1),\dots,\dots , (\alpha_k,\beta_k) )\in \randpolpair^k$. Then,
\begin{align}\label{e-randtopure}
\norm{P^{\alpha_1,\beta_1}\dots P^{\alpha_k,\beta_k}}_{\mathrm{H}} \leq
\max_{(\sigma_1,\tau_1),\dots, (\sigma_k,\tau_k) \in \purepolpair^k}
\norm{P^{(\sigma_1,\tau_1)}\dots P^{(\sigma_k,\tau_k)}}_{\mathrm{H}}  \enspace .
\end{align}
\end{lemma}                     %
\begin{proof}
  Using the integral representation~\eqref{e-integral}, we get
  \begin{align*}
    P^{\alpha_1,\beta_1}\dots P^{\alpha_k,\beta_k}=\int_{\purepolpair^k}
    P^{\sigma_1,\tau_1}\dots P^{\sigma_k,\tau_k}
    d\nu^{\alpha_1,\beta_1}(\sigma_1,\tau_1)
    \dots d\nu^{\alpha_k,\beta_k}(\sigma_k,\tau_k) \enspace .
  \end{align*}
  Since $\nu^{\alpha_1,\beta_1}\otimes \dots\otimes \nu^{\alpha_k,\beta_k}$ is a probability measure
  on $\purepolpair^k$, 
  using the convexity of the operator norm $\|\cdot\|_H$, we get~\eqref{e-randtopure}.
  Moreover, a compactness argument shows that the maximum in~\eqref{e-randtopure} is achieved.

\end{proof}
Now, \Cref{contraction_concu} follows by combining~\Cref{lemme_barycentre}, ~\Cref{D_k},
\Cref{thm_unichain} and~\Cref{prop-irr}. \hfill\qed
\begin{remark}\label{rk-explain}\todo{SG: added this remark}
  Combining~\Cref{D_k} and~\Cref{lemme_barycentre}, we see that the Shapley
  operator $T$ of a concurrent game is a contraction in Hilbert's seminorm
  of rate bounded by
  \begin{align}
  \max_{(\sigma_1,\tau_1),\dots, (\sigma_k,\tau_k) \in \purepolpair^k}
  \norm{P^{(\sigma_1,\tau_1)}\dots P^{(\sigma_k,\tau_k)}}_{\mathrm{H}}  \enspace .
  \label{newdob}
\end{align}
This should be compared with formula (3.6) of~\cite{FEDERGRUEN1978711},
which applies to the one player case. Unlike~\eqref{newdob},
the formula (3.6) involves the Dobrushin ergodicity coefficients of {\em rectangular}
matrices, of size $2n\times n$, arising from pairs of policies of a single player. The proof above, building on different
principles (semidifferentiability properties and convexity of the norm), leads
to an improved estimate -- showing that there is no need to consider pairs of policies of the same player.
  \end{remark}

\section{Proof of~\Cref{cor-complexity-concur}}
  From  \Cref{contraction_concu}, we know that $\shapleytheta^{k}$ is a
  contraction of rate $1-\theta^{k}$ in Hilbert's seminorm.
  Then, the bound on the number of iterations follows from~\Cref{th-terminates}
  from the concavity of the log function, which yields $|\log(1-\delta)|^{-1}\leq \delta^{-1}$ for all $0<\delta<1$,
  from $\hilbert{T(0)}\leq 2\|r\|_\infty$, and from the fact that, since $\chi(\shapleytheta)=(1-\theta)\chi(\shapley)$,
we need to apply~\Cref{RVI} with a precision of $(1-\theta)\epsilon$ to $\shapleytheta$, to get a final precision of
  $\epsilon$ for the value of the original game.
  \hfill\qed

  \section{Proof of~\Cref{th-uni-turn}}
\label{sec-proof-uni-turn}
  Let $x^*,\alpha,\beta$ be the vector and scalars returned when applying
  \Cref{RVI} to the operator $\shapleytheta$. Let $\mu$ be the value of the turn-based
  mean-payoff game with operator $T$. Since the game is unichain, there exists
  a vector $u\in\R^n$ such that $T(u)=\mu e +u$, so that $\shapleytheta(u) = \lambda e +u$
  with $\lambda = (1-\theta) \mu$.
  Then, by~\eqref{e-ub},
  \[
  \shapleytheta (x^*) \leq (\beta  +\eta)e +x^* \enspace,
  \]
 and $\beta + \eta \leq \lambda + \epsilon$.

    We shall make use of the following lemma.
    \begin{lemma}\label{lem-sep-turn}
      Let $\sigma',\tau'$ and $\sigma'',\tau''$ denote two pairs of 
      policies, that yield distinct values $\mu'$, $\mu''$.
      Then,
      \begin{align}\label{e-sep-turn}
      |\mu'-\mu''|> (nM^{n-1})^{-2} \enspace .
      \end{align}
    \end{lemma}
    \begin{proof}
Denoting by $\pi$ the unique invariant measure of the matrix $P^{{\sigma'},{\tau'}}$,
      we have $\mu = \pi r^{\sigma',\tau'}$. By~\Cref{mateusz}, and using~\Cref{a-integer},
      we get that $\mu$ is a rational number whose denominator at most $nM^{n-1}$.
      Noting that the difference between two rational numbers of denominators at most $q$
      is at least $1/q(q-1)>1/q^2$, we get~\eqref{e-sep-turn}.
      \end{proof}

Now, by definition of $\sigma^*$, for all pure policies $\tau$ of player Max,
  \begin{align}
  r^{\sigma^*\tau}+ \theta x^* + (1-\theta)P^{\sigma^*\tau}x^* \leq (\beta  +\eta)e +x^* \enspace. \label{e-sub}
  \end{align}
  Let $\mu'$ denote the mean-payoff associated
  to the pair of policies $\sigma^*,\tau$, $\lambda'=(1-\theta)\mu'$,
  and let
  $\pi'$ denote the unique invariant measure or
  $P^{\sigma^*\tau}$, so that $\pi'P^{\sigma^*\tau}=\pi'$
  and $\pi'e=1$. Left multiplying~\eqref{e-sub} by $\pi'$,
  and noting that the contribution of $x^*$ cancels, we deduce that
  that $\lambda'= (1-\theta)\mu'\leq \beta + \eta \leq \lambda + \epsilon=(1-\theta)\mu+\epsilon$.
  Since we chose $\epsilon=(1-\theta)(nM^{n-1})^{-2}$, it follows from~\eqref{e-sep-turn} that
  $\mu'\leq \mu$, which shows that $\sigma^*$ is an optimal policy of Min.
  The dual argument allows one to show that $\tau^*$ is an optimal policy of player Max.

  Finally, the estimate~\eqref{e-neweps} of the number of iterations is gotten by combining~\Cref{th-terminates}
and~\Cref{contraction_concu}.
\section{Proof of~\Cref{km-mult}}\label{sec-proof-km-mult}
We first show the following lemma.\todo{SG: lemma and proof added}
\begin{lemma}\label{lem-comb}
  Let $\mathcal{A}_l$ denote the $l$-ambiguity of the game, and let
  $\mathcal{A}= \max_{1\leq l\leq \irr}\mathcal{A}_l^{1/l}$.
  Then, we have $\mathcal{A}_l\leq n^{l-1}W^l$, and
  $\mathcal{A}\leq n^{1-1/k}W$ where $k\coloneqq \irr$.
  Moreover, for all $s>0$, and for all policies $\sigma_1,\tau_1,\dots,\sigma_{k},\tau_{k}$,
  the entries of the matrix
  \[
\mathcal{M}(s)\coloneqq (s I + M^{\sigma_1,\tau_1})\dots (s I + M^{\sigma_{k},\tau_{k}})
  \]
  do not exceed $(s+\mathcal{A})^k$.
\end{lemma}
\begin{proof}
  Let $J$ denote the $n\times n$ matrix whose entries are identically $1$, and observe that $J^2=nJ$. Then,
  \( M^{\sigma_1\tau_1}\dots M^{\sigma_l\tau_l}\leq (WJ)^l=n^{l-1}JW^l \). It follows
  that $\mathcal{A}_l\leq n^{l-1}W^l$ and $\mathcal{A}\leq n^{1-1/k}W$. 
  By expanding the product defining $\mathcal{M}(s)$, we get
  \begin{align*}
    (\mathcal{M}(s))_{d,d'}
     & = \sum_{0\leq l\leq k} \sum_{i_1<i_2<\dots<i_l}
    s^{k-l}[  M^{\sigma_{i_1},\tau_{i_1}}
      \dots
      M^{\sigma_{i_l},\tau_{i_l}}    ]_{d,d'}\\
    & \leq \sum_{0\leq l\leq k} {k \choose l} s^{k-l}\mathcal{A}_l
    \leq \sum_{0\leq l\leq k} {k \choose l} s^{k-l}\mathcal{A}^l
    = (s+\mathcal{A})^k \enspace .
    \end{align*}
\end{proof}
It will be convenient to define:
\[
  [\FmKM(x)]_d  = \min_{t\in \vertexsetT,(d,t)\in\edges} \max_{p\in\vertexsetP,(t,p)\in\edges} \big( \vartheta x_d + \sum_{d'\in \vertexset{D},(p,d')\in\edges} m_{p,d'} x_{d'}\big),\qquad
  \forall d\in \vertexsetD
  \enspace ,
  \]
  so that $\shapleymKM = \log \circ \FmKM \circ \exp$,
  where the $\log$ and $\exp$ operations are taken entry-wise.
 It follows that $\shapleymKM^k = \log \circ \FmKM^k \circ \exp$.

We first observe that $\FmKM$ and $\shapleymKM$ are semi-differentiable. 
Let $u$ and $h$ be two vectors of $\RR^n$. Since $\log'_u(h) = \diag(u)^{-1}h$ and $\exp'_u(h) = \diag(e^u)h$, we have:
\[ (\shapleymKM^k)'_u(h) = \diag(\FmKM(e^u))^{-1} [(\FmKM^k)'_{e^u} \diag(e^u)h] \]
And, we know that for all vectors $v$ and $h'$:
\[ (\FmKM^k)'_{v} h' = (\FmKM')_{\FmKM^{k-1}(v)} \circ \cdots \circ (\FmKM')_v(h')\]
Moreover, for all $v$, $h'$, there exists $\sigma$ and $\tau$ such that:$(\FmKM')_v (h') = (m I + M^{\sigma, \tau})h'$. Hence, there exists, $\sigma_1, \tau_1, \cdots, \sigma_k, \tau_k$ such that:
\[ (\FmKM^k)'_{e^u} \diag(e^u)h = (\vartheta I + M^{\sigma_1, \tau_1}) \cdots (\vartheta I + M^{\sigma_k, \tau_k})\diag(e^u)h \]
We denote $\mathcal{M} = (\vartheta I + M^{\sigma_1, \tau_1}) \cdots (\vartheta I  + M^{\sigma_k, \tau_k})$, then :\
$(\shapleymKM^k)'_u(h) = \diag(\FmKM(e^u))^{-1} [\mathcal{M} \diag(e^u)h]$.
Besides, $\FmKM^k(e^u) = \mathcal{M}e^u$.
Hence:
\[ (\shapleymKM^k)'_u(h) = \diag(\mathcal{M}(e^u))^{-1} [\mathcal{M} \diag(e^u)h] \enspace .\]
Recall that Hilbert's projective metrix is defined by\todo{SG: added def of Hilbert's metric, and corrected a few typos in this part}
\[
d_{\mathrm{H}}(x,y)\coloneqq\|\log x-\log y\|_H,\qquad \forall x,y\in \R_{>0}^n \enspace .
\]
Since the matrix $\mathcal{M}$ has positive entries, the Birkhoff-Hopf theorem
yields
\[ d_{\mathrm{H}} (\mathcal{M}x, \mathcal{M}y) \leq \tanh(\frac{\Delta}{4}) d_{\mathrm{H}}(x, y)\enspace,\]
where
\[
\Delta= \max_{i,j} \|\log \calm_{i\cdot} -\log \calm_{j\cdot}\|_H \enspace,
\]
is the diameter of the the set $\calm(\R_{>0}^n)$ in Hilbert's projective metric,
see~\cite[Appendix~A]{nussbaumlemmens} for more information.
Hence:
\[ \norm{\log(\mathcal{M}e^u) - \log(\mathcal{M}e^v)}_{\mathrm{H}} \leq \tanh(\frac{\Delta}{4})\norm{u-v}_{\mathrm{H}}\]
Using the fact that \[(\log \circ \mathcal{M} \circ \exp)'_u h = \lim_{t \to 0} \frac{\log \circ \mathcal{M} \circ \exp(u+th) - \log \circ \mathcal{M} \circ \exp(u)}{t}\] 
We have: \[ \norm{(\log \circ \mathcal{M} \circ \exp)'_u h}_{\mathrm{H}} \leq \tanh(\frac{\Delta}{4})\norm{h}_{\mathrm{H}}. \]
Thus, \[ \norm{(\shapleymKM^k)'_u(h)}_{\mathrm{H}} \leq \tanh({\frac{\Delta}{4}}) \norm{h}_{\mathrm{H}} \enspace .\]
Since every non-zero entry of $\mathcal{M}$ is bounded below by $\vartheta=\mbar$,
the minimal entry of $\mathcal{M}$
is bounded below by $\vartheta^k$.
Moreover, by~\Cref{lem-comb},
every entry of $\mathcal{M}$ is bounded above by $(\vartheta  +\mathcal{A})^k$.
\todo{SG: bound modified taking into account the ambiguity}
It follows that
\[ e^{\Delta} = \max_{i, j, i', j'} \frac{\mathcal{M}_{i, j}\mathcal{M}_{i', j'}}{\mathcal{M}_{i', j}\mathcal{M}_{i, j'}}
\leq (1+\mathcal{A}/\vartheta)^{2k}=\mcst^2
\enspace .
\]
Then $\tanh{\frac{\Delta}{4}} = \frac{e^{\Delta/2} - 1}{e^{\Delta/2} + 1} \leq \frac{\mcst - 1}{\mcst + 1}$.
Thus,
\[\norm{\shapleymKM^k(x) - \shapleymKM^k(y)}_{\mathrm{H}} \leq \frac{\mcst - 1}{\mcst + 1}\norm{x-y}_{\mathrm{H}}.\]

  \begin{remark}
We note that \Cref{RVI} may be implemented by considering the iteration $y\coloneqq \FmKMtilde(y)/\mathsf{t}(\FmKMtilde(y))$, where $\FmKMtilde$ is an approximation of the operator $\FmKM$, such that $\|\log \FmKM(x)-\log\FmKMtilde(x)\|_\infty \leq \eta$,  avoiding 
unecessary evaluations of $\log$ and $\exp$ at every iteration. Then, the stopping condition should be replaced by $\norm{\log(y)-\log(\FmKMtilde(y))}_{\mathrm{H}} \leq \epsilon$.
\end{remark}
\section{Proof of~\Cref{th-irr-entropy}}
We follow the same method as in the Proof of~\Cref{th-uni-turn}, but now exploiting the
separation bound of~\Cref{th-sepentropy}.

If an entropy game is irreducible, then, the recession function $\hat{T}$
has only fixed points that belong to the diagonal $\R e$, and then,
it follows from~\cite[Th.~9 and~13]{arxiv1} that the eigenproblem $T(u) =\nu e +u$
with $\nu\in \R$ and $u\in \R^n$ is solvable. Moreover, $\nu = \log\mu$,
where $\mu$ is the value of the entropy game.
Then, the non-linear eigenvalue $\lambda$ solution of the eigenproblem
$\shapleymKM(u) = \lambda e +u $ is given by $\log(\mbar+\mu)$.
Moreover, \Cref{RVI} applied to the operator $\shapleymKM$ returns a vector $x^*$
and scalars $\alpha,\beta$ such that
\[
(\alpha-\eta) e + x^* \leq \shapleymKM(x^*) \leq (\beta+\eta) e + x^* 
\]
where
\[
\log (\mbar+\mu) - \epsilon\leq\alpha -\eta \leq  \beta + \eta \leq \log (\mbar+\mu) + \epsilon
\]

We now observe that if two pairs of strategies yield distinct values $\mu'$ and $\mu''$ in an entropy game,
then, applying the Taylor formula with exact remainder,  together with~\Cref{th-sepentropy},
\begin{align}
|\log(\mbar+\mu') -\log (\mbar+\mu'')|\geq ((\mbar+W)\nu_{n})^{-1} \enspace.\label{e-sepeff}
\end{align}
Let $\sigma^*$ denote a policy of Despot obtained by selecting, for each node $d\in \vertexsetD$,
a minimizing action in the expression
\[
\min_{t\in \vertexsetT,(d,t)\in\edges} \max_{p\in\vertexsetP,(t,p)\in\edges} \big( \vartheta x^*_d + \sum_{d'\in \vertexsetD,(p,d')\in\edges} m_{p,d'} \exp(x^*_{d'})\big),
\]
and similarly, let $\tau^*$ denote a policy of Tribune obtained by selecting, for each node
$t\in\vertexsetT$, a maximizing action in the inner expression above.

Define, for all pairs of policies $\sigma,\tau$, the map
\[
\shapleymKM^{\sigma,\tau} (x)\coloneqq \log \big((\vartheta I + M^{\sigma,\tau})\exp(x)\big) \enspace .
  \]
  We deduce from $\shapleymKM(x^*) \leq (\beta+\eta)e +x^*$ that
  $\shapleymKM^{\sigma^*,\tau}(x^*)  \leq (\beta+\eta) e +x^*$, and if $\mu'$
  is the value of the pair of policies $(\sigma^*,\tau)$, this
  entails that $\log(\mbar + \mu') \leq \beta +\eta
  \leq \log(\mbar +\mu)+\epsilon$. Since we chose
  $\epsilon= (1+(\mbar+W)\nu_n)^{-1}< ((\mbar+W)\nu_{n})^{-1}$, using~\Cref{e-sepeff},
 we deduce that $\sigma^*$ guarantees to Despot
 a value which does not exceed $\mu$. A dual argument
 shows that $\tau^*$ guarantees to Tribune a value of at least $\mu$.
It follows that $\sigma^*$ and $\tau^*$ are optimal policies.

Finally, the number of iterations given in~\Cref{th-irr-entropy} is obtained by applying~\Cref{th-terminates},
with $\gamma = (\mcst-1)/(\mcst+1)$, observing that $(\mcst-1)/(\mcst+1)\leq 1/(1+2 \mcst^{-1})$,
and using the concavity of the log map to deduce that
$1/|\log \gamma|  \leq 1/\log(1+2\mcst^{-1})\leq 1/(2\mcst^{-1})=\mcst/2$.\todo{SG: proof of bound added and bound corrected}
\end{document}